\documentclass[a4paper,11pt,english]{smfart}
\usepackage[applemac]{inputenc}
\usepackage[english]{babel}
\usepackage{amsfonts}
\usepackage{amsmath}
\usepackage{amsthm}
\usepackage{hyperref}
\usepackage{latexsym}
\usepackage{array}
\usepackage{amssymb}
\usepackage{enumerate}
\usepackage{smfthm}
\usepackage{graphicx}
\usepackage{color}
\usepackage{stmaryrd}
\newcommand{\ligne}{\vspace{1\baselineskip}}
\newcommand{\ph}{\phantomsection}
\newcommand{\cal}{\mathcal}

\newcommand{\R}{\mathbb  R}

\newcommand{\N}{\mathbb  N}

\newcommand{\M}{  \mathcal{M}  }
\newcommand{\eps}{\varepsilon}
\renewcommand{\epsilon}{\varepsilon}
\newcommand{\e}{  \text{e}   }

\newcommand{\V}{  \mathcal{V}   }
\newcommand{\D}{  \mathcal{D}   }

\newcommand{\dis}{\displaystyle}

\renewcommand{\a}{  \alpha   }
\renewcommand{\b}{  \beta   }
\newcommand{\s}{  \sigma   }

\renewcommand{\phi}{  \varphi  }
\renewcommand{\L}{  \Lambda   }

\newcommand{\<}{  \langle   }

\renewcommand{\>}{  \rangle   }

\renewcommand{\S}{  \mathbb{S}  }

\numberwithin{equation}{section}
\theoremstyle{plain}

\newtheorem{assumption}{Assumption}

\def\beq{\begin{equation}}   \def\eeq{\end{equation}}
\def\bea{\begin{eqnarray}}  \def\eea{\end{eqnarray}}

\renewcommand{\theequation}{\thesection.\arabic{equation}}
\newcounter{hran} \renewcommand{\thehran}{\thesection.\arabic{hran}}

\def\bmini{\setcounter{hran}{\value{equation}}
   ~\refstepcounter{hran}\setcounter{equation}{0}
    \renewcommand{\theequation}{\thehran\alph{equation}}\begin{eqnarray}}

\def\bminiG#1{\setcounter{hran}{\value{equation}}
\refstepcounter{hran}\setcounter{equation}{-1}
\renewcommand{\theequation}{\thehran\alph{equation}}
\refstepcounter{equation}\label{#1}\begin{eqnarray}}

\def\cc{{\cal C}}

\def\mm{{\cal M}}
\def\ss{{\cal S}}\def\vv{{\cal V}}
 
\def\R{\mathbb R}\def\N{\mathbb N}
  \def\S{\mathbb S}   
\def\D{\partial}\def\eps{\varepsilon}\def\phi{\varphi}

\def\norm#1{\left\Vert#1\right\Vert}

\def\set#1{\big\{#1\big\}}
\def\seq#1{\left<#1\right>}
\def\sep#1{\left(#1\right)}
 
\def\defegal{\stackrel{\text{\rm def}}{=}}

\def\Uzero{U_{0}}
\def\Vzero{V_0}

\def\Ve{V_e}
\def\V{V}

\def\Uinf{U_{\infty}}
\def\Vinf{V_{\infty}}

\def\mme{{\cal M}_{e}}

\def\mminf{{\cal M}_\infty}

\def\Be{{ B}_e}

\def\Kinf{K_\infty}
\def\Einf{E_\infty}
\def\Ke{K_{e}}

\def\Hess{\text{\rm{Hess}}}
\def\Id{\textrm{Id}}

\definecolor{gr}{rgb}   {0.,   0.69,   0.23 }
\definecolor{bl}{rgb}   {0.,   0.5,   1. }
\definecolor{mg}{rgb}   {0.85,  0.,    0.85}
\definecolor{yl}{rgb}   {0.8,  0.7,   0.}
\definecolor{or}{rgb}  {0.7,0.2,0.2}
\newcommand{\Bk}{\color{black}}

\author{ Fr\'ed\'eric H\'erau}
\address{Laboratoire de Math\'ematiques J. Leray, UMR  6629 du CNRS, Universit\'e de Nantes,
2, rue de la Houssini\`ere,
44322 Nantes Cedex 03, France}
\email{frederic.herau@univ-nantes.fr}
\author{ Laurent Thomann }
\address{
Laurent Thomann\\
Institut  \'Elie Cartan, Universit\'e de Lorraine, B.P. 70239,
F-54506 Vandoeuvre-l\`es-Nancy Cedex, France}
\email{laurent.thomann@univ-lorraine.fr}

 \title[On global existence for Vlasov-Poisson-Fokker-Planck]{On global existence and trend to the equilibrium for the Vlasov-Poisson-Fokker-Planck system with exterior confining potential}

\begin{document}
\frontmatter

 \begin{abstract}
 We prove a global existence result with initial data of low regularity, and prove the trend to the equilibrium  for the Vlasov-Poisson-Fokker-Planck system with small non linear term but with a possibly large exterior confining potential in dimension $d=2$ and $d=3$.  The proof relies on a fixed point argument using sharp estimates (at short and long time scales) of the semi-group associated to the Fokker-Planck operator, which were obtained by the first author.    \end{abstract}
\thanks{F.H. is supported by the grant "NOSEVOL" ANR-2011-BS01019-01. L.T. is  supported   by the  grant  ``ANA\'E'' ANR-13-BS01-0010-03. \\
The authors warmly thank Jean Dolbeault for enriching discussions and are grateful  to  Laurent Di Menza who was at the origin of this collaboration.}
\maketitle
\mainmatter

 \section{Introduction and results}

 \subsection{Presentation of the equation}

Let $d=2$ or $d=3$. We consider the  Vlasov-Poisson-Fokker-Planck   system (VPFP for short)  with external potential,
  which reads, for $(t,x,v)\in [0,+\infty)\times \R^{d}\times \R^{d}$

\begin{equation}
  \label{vpfp}
  \left\{
  \begin{aligned}
   &  \D_{t}f+v.\D_{x}f -(\eps_{0} E+ \D_x V_e).\D_{v}f  -\gamma \D_{v}.\left(\D_{v}+v\right)f=0, \\
   & E(t,x)= -  \frac{1}{|\S^{d-1}|}   \frac{x}{|x|^d}\star_x \rho(t,x),  \ \     \ \    \text{where }\;\; \rho(t,x) = \int     f(t,x,v)dv, \\
    &   f(0,x,v) = f_0(x,v),
    \end{aligned}
    \right.
\end{equation}
where $x\mapsto V_e(x)$ is a given smooth confining potential (see Assumption~\ref{assumption1} below). The constant $\eps_0 \in
\R$ is the total charge of the system and  in \Bk the sequel we assume that either $\eps_0>0$ (repulsive case) or $\eps_0<0$ (attractive case) in the case $d=3$. The constant $\gamma>0$ is the friction-diffusion coefficient, and for simplicity we will take $\gamma=1$.

The unknown $f$ is the distribution function of the particles.
We assume that $f_0 \geq 0$ and that $\dis \int f_{0}(x,v)dxdv=1$, it is then easy to check that   once a good existence theory is given,   these properties are preserved, namely that for all $t\geq0$
\begin{equation*}
f \geq 0 \qquad \textrm{ and }\qquad \int  f(t,x,v)dxdv= 1,
\end{equation*}
and we refer to Section~\ref{basic} for more details and other basic results.\medskip

 This equation is a  model for a plasma submitted to an external confining electric field (in the repulsive case) and also a model for gravitational systems (in the attractive case).
 When there is no external potential ($V_e=0$), the equation has been exhaustively studied. First existence results were obtained by Victory and O'Dwyer in 2d~\cite{OV90} and by Rein and Weckler~\cite{RW92} in 3d for small data. Bouchut~\cite{Bou93} showed that the equation  is globally well-posed in 3 dimensions using the explicit kernel.
  The long time behavior (without any rate)  has  been studied with or  without external potential \Bk
 by Bouchut and Dolbeault in~\cite{BD95},  Carillo, Soler and Vazquez~\cite{CSV96}, and also by Dolbeault in~\cite{Dol99}.\medskip

When there is a confining potential, arbitrary polynomial trend to the equilibrium was established  in~\cite{DV01} where a first  notion of hypocoercivity~\cite{villani}  was developed and used later to the full model~\cite{DV05}. The exponential trend to the equilibrium  was
shown in the linear case (the Fokker-Planck equation) for a general external confining potential in~\cite{HN04} (see also~\cite{HelN04}).
So far, in the non-linear case,  there is no general result  about exponential trend to the equilibrium. In the case of the torus (and $V=0$), the strategy of Guo  can be applied to many models (see  e.g.~\cite{Guo02,Guo03a,Guo03b}). In the case when the potential is explicitly given by $V_e(x)=C |x|^{2}$, a recent result with small data is given in~\cite{HJ13}, following the micro-macro strategy of Guo.\medskip

In all previous cases (torus, $V_e = 0$ or polynomial of order $2$), mention that one can  compute explicitly the Green function of the Fokker-Planck operator and also that exact computations can be done thanks to vanishing commutators. Here instead we will rely on estimates (in short and long time) of the linear solution of the Fokker-Planck operator  obtained by the first author in~\cite{her06}, and our approach allows us to deal with a large class of confining potentials $V_e$. Indeed, in~\cite[ Theorem 1.3 \Bk]{her06} a first exponential trend to the equilibrium result for a VPFP type model was  given, but only for a mollified  non-linearity. We will prove here a global existence result in the full VPFP case,  with trend to equilibrium assuming that the initial condition $f_{0}$ is localised and has some Sobolev regularity and under the assumption that the electric field is perturbative in the sense that $|\eps_0|\ll 1$.\medskip

 Let us now precise our notations and hypotheses. We do not try to optimise the assumptions on the confining potential $V_e$ and first assume the following

 \begin{assumption}\ph\label{assumption1}
The potential $x\mapsto V_e(x)$ satisfies
\begin{equation*}
 e^{-\V_e} \in
\ss(\R^d), \ \ \ \text{with}  \ \ \ \V_e \geq 0 \ \  \textrm{ and
} \ \   \V_e'' \in W^{\infty,\infty}(\R^{d}).
\end{equation*}
\end{assumption}

Observe that the assumption $  \V_e\geq 0 $ can be relaxed by assuming that $V_e$ is bounded from below and adding to
it a sufficiently large constant.

 We now introduce the Maxwellian of the equation~\eqref{vpfp}
 \begin{equation} \label{defmme}
\mminf(x,v) = \frac{e^{-(v^2/2 + \V_e(x)+\eps_{0}\Uinf(x))}}{\int e^{-(v^2/2
+ \V_e(x)+\eps_{0}\Uinf(x))}dxdv},
\end{equation}
where $\Uinf$ is a solution of the following Poisson-Emden type equation
\begin{equation} \label{VVinf}
-\Delta \Uinf =    \frac{ e^{-(V_e+\eps_{0}\Uinf)}}{\int
e^{-(V_e(x)+\eps_{0}\Uinf(x))}dx}.
\end{equation}
Actually, one gets that under Assumption~\ref{assumption1} and $|\eps_{0}|$ small enough (assuming additionally that $\eps_{0}>0$ in the case $d=2$), the equation~\eqref{VVinf} has a
unique (Green) solution $\Uinf$ which belongs to $  W^{\infty,\infty}(\R^{d})$ uniformly w.r.t $|\eps_0|$ (see Propositions~\ref{corpeintro} and~\ref{corpeintro1} following results from~\cite{Dol91}).
 The Maxwellian  $\mminf$ is then  in $\ss(\R^{d}_{x}\times \R_{v}^{d})$ and is the unique $L^1$-normalised steady solution of equation\;\eqref{vpfp}.

In the case $d=2$ and $\eps_{0}<0$, existence and uniqueness of solutions to~\eqref{VVinf} are unclear, that's why we do not consider this case.\medskip

For convenience, we now introduce the effective potential at infinity
 \begin{equation} \label{defV}
 \Vinf \defegal V_e+\eps_{0}\Uinf \ \ \ \textrm{ so that } \; \;\mminf(x,v) = \frac{e^{-(v^2/2 + \Vinf(x))}}{\int e^{-(v^2/2
+ \Vinf(x))}dxdv}.
\end{equation}
The second assumption on $V_e$ is the following

 \begin{assumption}\ph\label{assumption2}
The so-called  Witten operator
$W= -\Delta_{x}+ {|\partial _{x} \Ve|^{2}}/{4}-{\Delta_{x}\Ve}/{2}$
has a spectral gap    in $L^{2}(\R^{d})$. We denote by $\kappa_0>0$ the minimum of this spectral gap and $ d/2$.
\end{assumption}

\begin{exem} As an example, we can check that if $V_e$ satisfies Assumption~\ref{assumption1} and is such that
$${|\D_x V_e(x)|\underset{|x| \longrightarrow \infty}{\longrightarrow} +\infty}$$
 then it satisfies also Assumption~\ref{assumption2}  since it has a compact   resolvent. \Bk
\end{exem}

 We introduce now the functional framework on which our analysis is done.
We consider  the weighted space  $B$ built from the standard $L^2$
space after conjugation with a half power of the Maxwellian
 \begin{equation} \label{defb2}
B \defegal \mminf^{1/2} L^2 =  \big\{ f \in \mathcal{S}'(\R^{2d}) \text{ s.t. }
f/\mminf \in L^2(\mminf dxdv)\big\}.
\end{equation}
We define the natural scalar product
\begin{equation*}
\<f,g\>=\int fg \mminf^{-1} dxdv,
\end{equation*}
and the corresponding norm
\begin{equation*}
\begin{split}
& \norm{f}_{B}^2= \<f,f\>=\int f^2 \mminf^{-1} dxdv.
\end{split}
\end{equation*}
Next, consider the Fokker-Planck operator associated to the potential $\Vinf$
 defined by
\begin{equation} \label{Ke}
\Kinf = v.\D_{x} -\D_x \Vinf(x).\D_{v}- \gamma \D_{v}.\left(\D_{v}+v\right).
\end{equation}
The last object we need before writing our equation in a suitable way
 is the limit electric field
\begin{equation*}
\Einf(x)=\partial_{x}\Uinf(x)=-  \frac{1}{|\S^{d-1}|}   \frac{x}{|x|^d}\star_x  \int \mminf(x,v) dv.
\end{equation*}
 With all the previous notations, the VPFP equation~\eqref{vpfp} can be rewritten
 \begin{equation}
  \label{vpfp1}
  \left\{
  \begin{aligned}
   &  \D_{t}f+\Kinf f=\eps_0(E-\Einf)\partial_{v}f, \\
   & E(t,x)    = -  \frac{1}{|\S^{d-1}|}   \frac{x}{|x|^d}\star_x \rho(t,x),  \ \ \ \ \text{where }\;\; \rho(t,x) = \int     f(t,x,v)dv, \\
    &   f(0,x,v) = f_0(x,v).
    \end{aligned}
    \right.
\end{equation}
We define  the \Bk  operator
 \begin{equation*}
 \Lambda_{x}^2=  - \D_x.\big(\D_x+ \D_x \Vinf \big)  +1 \Bk
 \end{equation*}
which is up to a conjugation with $\mminf^{1/2}$ the Witten operator introduced in Assumption~\ref{assumption2} but defined on $B$, and
 \begin{equation*}
 \Lambda_{v}^2 = - \D_v.(\D_v+v)  +1 \Bk,
 \end{equation*}
 which is again up to a conjugation   the harmonic oscillator in velocity. They both are non-negative selfadjoint unbounded operators in $B$.
 We also introduce
 \begin{equation*}
 \Lambda^2= -  \D_x.\big(\D_x+ \D_x \Vinf\big) -  \D_v.(\D_v+v) +1 = \Lambda_x^2 +  \Lambda_v^2  -1. \Bk
 \end{equation*}
 It is clear that 
 $$
 1 \leq \Lambda^2_{x },\; \Lambda^2_{v }  \leq \Lambda^2. 
 $$ \Bk
 As we mentioned previously,  if $V_e$ satisfies Assumptions~\ref{assumption1} and~\ref{assumption2}, then $\Vinf = V_e+ \epsilon_0 \Uinf$ also does, and we check in Subsection\;\ref{spectralgap} that the operator 
  $$ 
 -  \D_x.\big(\D_x+ \D_x \Vinf\big) -  \D_v.(\D_v+v)  = \Lambda^2-1 
 $$ \Bk
  has $0$ as single eigenvalue and a spectral gap bounded in $B$ which is, uniformly w.r.t\;$|\eps_0|$ small, bounded from below by  $\kappa_0/2$.  \medskip

In the sequel, we will need the anisotropic chain of  Sobolev spaces: for $\a,\beta\geq0$
\begin{equation}\label{defB}
B^{\a,\b} =B^{\a,\b}_{x,v}(\R^{2d})=\big\{f\in B \,:\;\;   \Lambda^{\a}_{x}  f\in B^{}\;\;\text{and}\;\; \Lambda^{\b}_{v}f\in B^{}\big\},
\end{equation}
and we endow this space by the norm
\begin{equation*}
 \|f\|_{B^{\a,\beta}}= 
\|\Lambda^{\a}_{x} f\|_{B^{}}+\|\Lambda^{\b}_{v} f\|_{B^{}}. \Bk
\end{equation*}
In the case $\alpha=\beta$ we simply define
\begin{equation*}
B^{\a}=B^{\a,\a}_{x,v}(\R^{2d})=\big\{f\in B \,:\;\;   \Lambda^{\alpha}f\in B^{}\big\},
\end{equation*}
with the norm
\begin{equation*}
 \|f\|_{B^{\a}}= \|\Lambda^{\alpha} f\|_{B^{}} \sim \|f\|_{B^{\a,\a}}. \Bk
\end{equation*}
We observe that $\mminf\in B^{\a,\b}$ for all $\a,\b\geq 0$, since we have $\mminf \in \mathcal{S}(\R^{2d})$.

 \subsection{Main results} We are now able to state our global well-posedness results.
\begin{theo}\ph \label{thm1}
Let $d=2$ and let $f_{0}\in B(\R^4) $. Assume moreover that Assumptions~\ref{assumption1} and~\ref{assumption2} are  satisfied. Then if $\eps_0>0$ is small enough, there exists  a unique global mild solution $f$ to~\eqref{vpfp} in the class
\begin{equation*}
 f\in \mathcal{C}\big(\,[0,+\infty[ \,;\,  B(\R^{4}) \big).
\end{equation*}
Moreover,  the following convergence to equilibrium holds true
\begin{equation*}
 \Vert f(t)-\M_{\infty} \Vert_{B^{}}\leq  C_{0}  e^{-\kappa_0 t/c},\qquad \forall t \geq 1,
\end{equation*}
and
\begin{equation*}
 \Vert E(t)-E_{\infty} \Vert_{L^{\infty}(\R^{2})}\leq C_{1}  e^{-\kappa_0 t/c},\qquad \forall t\geq 1.
\end{equation*}
\end{theo}

By mild, we mean $f$ and $E$ which satisfy the integral formulation of~\eqref{vpfp1}, namely
\begin{equation}
  \label{int}
  \left\{
  \begin{aligned}
   &  f(t) = e^{-\Kinf}f_0 +\eps_{0} \int_0^t e^{-(t-s)\Kinf}  (E(s)-\Einf)\D_v f(s)ds, \\
   & E(t)    = -  \frac{1}{|\S^{d-1}|} \frac{x}{|x|^d}\star_x \int f(t)dv.
    \end{aligned}
    \right.
\end{equation}
~

In the case $d=3$, we need to assume more regularity on the initial condition, but the known results about the uniqueness of the Poisson-Emden equation (see Subsection~\ref{sect32}) allow to consider also the case $\eps_{0}<0$.\medskip

 Denote by
\begin{equation}\label{hyp}
\dis U_{0}=\frac{1}{4\pi |x|} \star_x \int f_{0} dv,
\end{equation}
which is such that  $\dis \Delta U_{0}  =\int f_{0} dv $. Then
\begin{theo} \ph \label{thm2}
Let $d=3$  and     $1/2<a<2/3$. Assume     that $f_{0}\in B^{a,a}(\R^6)\cap L^{\infty}(\R^{6}) $ is such that $U_{0}\in W^{2,\infty}(\R^{3})$.  Assume moreover that Assumptions~\ref{assumption1} and~\ref{assumption2} are  satisfied. Then if $\vert \eps_0\vert$ is small enough, there exists  a unique global mild solution $f$ to~\eqref{vpfp} in the class
\begin{equation*}
 f\in \mathcal{C}\big(\,[0,+\infty[ \,;\,  B^{a,a}(\R^{6}) \big)\cap L_{loc}^{\infty}\big(\,[0,+\infty[ \,;\,  L^{\infty}(\R^{6}) \big).
\end{equation*}
Moreover, for all  $ a\leq  \a<2/3$ and $ a\leq   \beta<1$ such that $3\a-1<\beta <1$
\begin{equation}\label{reg}
 f  \in   \mathcal{C}\big(\,]0,+\infty[ \,;\,  B^{\alpha,\beta}(\R^{6})\big),
\end{equation}
and  the following convergence to equilibrium holds true
\begin{equation*}
 \Vert f(t)-\M_{\infty} \Vert_{B^{\alpha,\beta}}\leq  C_{0} e^{-\kappa_0 t/c},\qquad \forall t \geq 1,
\end{equation*}
and
\begin{equation*}
 \Vert E(t)-E_{\infty} \Vert_{L^{\infty}(\R^{3})}\leq C_{1}  e^{-\kappa_0 t/c},\qquad \forall t\geq 1.
\end{equation*}
\end{theo}

 In  the previous lines, the constants  $c,C_{1},C_{2}>0$ only depend on $\Vert \Vinf \Vert_{W^{2,\infty}}$  where $\Vinf$ was defined in~\eqref{defV}, on~$\Vert U_0 \Vert_{W^{2,\infty}}$ and on  $f_{0}$.

Notice that in Theorem~\ref{thm2}, the parameters $(\a,\b)$ can be chosen independently from $a$. It is likely that the assumption $a<2/3$ is  technical, but our proof needs that $\beta<1$ (see e.g. Corollary~\ref{coro25}).  Since in this work we focus on low regularity issues, we did not try to relax this hypothesis.

It is likely that the assumption made on $U_{0}$ is technical. It is needed here in order to guarantee that the linearised equation near $t=0$ enjoys reasonable spectral estimates. Observe (see Remark~\ref{rema1} for more details), that the assumption  $f_{0}\in B^{a,a}(\R^6)\cap L^{\infty}(\R^{6}) $ alone ensures that $U_{0}\in W^{2,p}(\R^{3})$ for any $2\leq p<+\infty$.

 An analogue of the regularizing estimate\;\eqref{reg} can also be obtained in Theorem\;\ref{thm1}. This can be proven by getting estimates in some spaces $B^{\alpha,\beta}_{x,v}$ as in the proof of Theorem\;\ref{thm2} (see Section\;\ref{sect5}). We did not include it here in order to simplify the argument.\Bk
 \medskip

The proof uses estimates of $\e^{-t\Kinf}$ in the space $B^{}$, obtained in~\cite{her06} by  the first author. Theorem\;\ref{thm2} extends~\cite[Theorem 1.3]{her06} where he considered a regularised version of the electric field $E$ in~\eqref{vpfp}, which was so that $E(t)\in L^{\infty}(\R^3)$ for any $f\in B$. Here we tackle this difficulty by using the Sobolev regularity of $f$ and a gain given by the integration in time. The proof relies on   a fixed point argument in a space based on $B^{\alpha,\b}$ in the $(x,v)$ variables, and allowing an exponential decay in time.\medskip

As a consequence of Theorems~\ref{thm1} and~\ref{thm2}, we directly obtain the exponential decay of the relative entropy. Let us define
\begin{equation*}
 H(f(t),\mminf)= \iint f(t) \ln \Big(\frac{f(t)}{\mminf}\Big)dxdv,
\end{equation*}
then

\begin{coro}
Let $d=2$ or $d=3$. Then under the assumptions of Theorem\;\ref{thm1} or Theorem\;\ref{thm2}, the corresponding solution $f$ of~\eqref{vpfp} satisfies
\begin{equation*}
 0\leq H(f(t),\mminf) \leq C   e^{-\kappa t/c},
\end{equation*}
where $C,c>0$ only depend on second order derivatives of $\Ve +\eps_0\Uinf$  and on $f_{0}$.\Bk
\end{coro}

We refer to~\cite[Corollary 1.4]{her06} for the proof of this result.

\subsection{Notations and plan of the paper}

 \begin{enonce*}{Notations}
 In this paper $c,C>0$ denote constants the value of which may change
from line to line. These constants will always be universal, or uniformly bounded with respect to the other parameters.
\end{enonce*}

The rest of the paper is organised as follows. In Section\;\ref{sect2} we prove some linear estimates on $\e^{-tK}$ (where $K$ is a generic linear Fokker-Planck operator). In Section\;\ref{sect3} we gather some estimates on solutions of~\eqref{vpfp}. Finally,  Sections~\ref{sect4} and~\ref{sect5} are devoted to the proofs of Theorems~\ref{thm1} and~\ref{thm2}   with  fixed points arguments.

\section{Semi-group estimates}\label{sect2}
In this section, we denote by $V$ a generic potential satisfying  Assumptions~\ref{assumption1} and~\ref{assumption2}. We also denote by $K$ the associated generic  linear \Bk  Fokker-Planck operator
$$
K = v.\D_{x}f -\D_x V.\D_{v}  - \D_{v}.\left(\D_{v}+v\right).
$$
Similarly, the   operators   $\Lambda_x^2 = -\D_x(\D_x + \D_x V)+1$, $\Lambda^2 = \Lambda_x^2 + \Lambda_v^2 -1$,  \Bk the normalized Maxwellian $\mm(x,v) =
e^{-(V(x) + v^2/2)}$ and  spaces of type $B^{\a, \b}$ are built with respect to this generic potential~$V$. For convenience, we also denote by $X_0 = v.\D_{x}f -\D_x V.\D_{v}$.

The aim of this section is to state some estimates of $e^{-tK}$ in $B-$type norms. These are  consequences of\;\cite{her06}. In all the following we pose
 $$
 \kappa=\kappa_0/C_0,
 $$
 where $\kappa_0$ is the spectral gap of the operator $W$ defined in Assumption~\ref{assumption2} (with $V$ as a potential), and\;$C_0$ is a large constant depending only on derivatives of $V''$ 
 explicitly given   in~{\cite[Theorem\;0.1]{HN04}}. \Bk

The operator $K$ is maximal accretive in $B$ (see e.g.~\cite[ Theorem\;5.5]{HelN04}). This enables us to define $\e^{-tK}$ and to prove that
\begin{equation}\label{maxi}
\norm{e^{-tK}}_{B \rightarrow B} \leq 1.
\end{equation}
 Following~\cite[ Theorem 3.1]{HN04}, operator   $\e^{-tK}\longrightarrow Id$ when $t\longrightarrow 0$, strongly in $B^{a,a}$ for any ${a\geq0}$.
Observe that all the estimates in this section are independent of the dimension $d$.
 For a complete analysis of the linear Fokker-Planck operator we refer to~\cite{HN04} or~\cite{HelN04}. We now give some regularizing estimates for the semi-group associated to $K$, in the spirit of~\cite[ Section 3\Bk]{her06}.

\begin{prop} \ph \label{prop21} There exists $C>0$ so that  for all $\a, \b \in [0,1]$  and   all $t > 0$
\begin{equation} \label{stxdecay}
\|\L^{\a}_{x} \e^{-tK}\|_{B^{}\to B^{}}\leq C(1+t^{-3\a/2})
,\quad  \| \e^{-tK}\L^{\a}_{x}\|_{B^{}\to B^{}}\leq C(1+t^{-3\a/2})
\end{equation}
and
\begin{equation} \label{stvdecay}
\|\L^{\b}_{v} \e^{-tK}\|_{B^{}\to B^{}}\leq C(1+t^{-\b/2}),\quad  \| \e^{-tK}\L^{\b}_{v}\|_{B^{}\to B^{}}\leq C(1+t^{-\b/2}).
\end{equation}
In the previous bounds, the  constant \Bk $C$ only  depends
 on a finite number of derivatives of $V$. 
\end{prop}

\begin{rema} \ph Note that the exponents $1/2$ in~\eqref{stvdecay}  and $3/2$ in~\eqref{stxdecay} when $\alpha = 1$  are optimal at least in the  case $V= 0$ \Bk and in the case when 
$V$ is a  definite quadratic form in $x$. This can be checked since in these  both cases,
the Green kernel of $e^{-tK}$   is explicit. In the case $V=0$ we refer  to\;\cite{Bou93}, and when $V$ is quadratic, we refer to the   general Mehler formula given in~\cite[ Section 4 \Bk]{Hor95}.
\end{rema}

\begin{proof}[Proof of Proposition~\ref{prop21}]
We first prove the estimate~\eqref{stvdecay}.  In~\cite[Proposition\;3.1]{her06}, reinterpreted in our framework, reads
\begin{equation} \label{her06v}
\norm{(\D_v+v) e^{-tK}}_{B \rightarrow B} \leq  C(1+t^{-1/2}).
\end{equation} \Bk
For $f$ a solution of the equation 
$$ \partial_t f+ K f=0,\quad f(t=0)=f_0,$$
with  normalized \Bk initial condition $f_0 \in \cc_0^\infty$, and using the regularization property of $e^{-tK}$, we have for $t>0$
 \begin{eqnarray*}
\norm{ \Lambda_v f(t)}_{B \rightarrow B}^2
& =& \<\Lambda_v^2 f(t),f(t)\> \\
& =& \norm{(\D_v+v)f(t)}_{B \rightarrow B}^2 + \norm{f(t)}_{B \rightarrow B}^2 \\
& \leq & C(1+t^{-1/2})^2+1 \\
& \leq  &C'(1+t^{-1/2})^2.
\end{eqnarray*} \Bk
Using that $B^{0,0}_{x,v} = B$ and~\eqref{maxi}, we therefore have that
$$
\norm{e^{-tK}}_{B \rightarrow B^{0,1}_{x,v}}^2  \leq C'(1+t^{-1/2}), \ \ \
\norm{e^{-tK}}_{B \rightarrow B^{0,0}_{x,v}}^2  \leq C''
$$
and by interpolation we get that for all $0\leq \beta\leq 1$
$$
\norm{e^{-tK}}_{B \rightarrow B^{0,\beta}_{x,v}}^2  \leq C(1+t^{-1/2})^\beta \leq C_b(1+t^{-\beta/2})
$$
which reads
$$
\|\L^{\b}_{v} \e^{-tK}\|_{B^{}\to B^{}}\leq C_{\b}(1+t^{-\b/2})
$$
which is the first result. For the converse estimate, we use that $K^*$, the adjoint of $K$ in $B$ given by $K^* = -X_0 -\D_v.(\D_v+v)$, has the same properties as $K$ so that for all $t>0$,
$$
\|\L^{\b}_{v} \e^{-tK^*}\|_{B^{}\to B^{}}\leq C'_{\b}(1+t^{-\b/2}).
$$
 Taking  \Bk the adjoints of this yields
$$
\| \e^{-tK}\L^{\b}_{v}\|_{B^{}\to B^{}}\leq C'_{\b}(1+t^{-\b/2}).
$$

Concerning the estimates involving $\Lambda_x$, the proof is exactly the same
as the preceding one with  $\Lambda_v$ replaced by $\Lambda_x$, $\beta$ replaced
by $3\alpha$, $-\D_v.(\D_v + v)$ replaced by $-\D_x.(\D_x + \D_x V(x))$ and using the result   from~\cite[Proposition\;3.1]{her06} \Bk
$$
\norm{(\D_x + \D_x V(x)) e^{-tK}}_{B \rightarrow B} \leq  C(1+t^{-3/2}),
$$
instead of~\eqref{her06v}.
This concludes the proof.
\end{proof}

From Proposition~\ref{prop21}, it is easy to deduce the following

\begin{coro} \ph \label{coro23}
Let $\a, \b \in [0,1]$. Then
\begin{equation*}
 \| \Lambda^{\a}_{x}\e^{-(t-s)K} \Lambda^{1-\b}_{v} \|_{B \to B } \leq  C\big((t-s)^{-1/2+\b/2-3\a/2}+1\big) ,
\end{equation*}
and
\begin{equation*}
 \| \Lambda^{\b}_{v}\e^{-(t-s)K} \Lambda^{1-\b}_{v} \|_{B \to B } \leq  C\big((t-s)^{-1/2}+1\big) .
\end{equation*}
\end{coro}
\begin{proof}
We only prove the first statement, the second is similar. By~\eqref{dx},~\eqref{dv} and also using Remark~\ref{divide} we have
 \begin{eqnarray*}
 \| \Lambda_{x}^{\a}\e^{-(t-s)K} \Lambda^{1-\b}_{v} \|_{B \to B }& \leq &
  \| \Lambda^{\a}\e^{-(t-s)K} \Lambda^{1-\b}_{v} \|_{B \to B }   \nonumber    \\
& \leq &  \| \Lambda^{\a}\e^{-(t-s)K/2}   \|_{B } \| \e^{-(t-s)K/2} \Lambda^{1-\b}_{v} \|_{B }\nonumber \\
 &\leq&  C\big((t-s)^{-1/2+\b/2-3\a/2}+1\big) ,
\end{eqnarray*}
which was the claim.
\end{proof}

We define
$$
B^{\perp} = \Big\{ f \in B \;\;s.t.\;\;\seq{f,\mminf}=\int fdxdv=0\Big\}
$$
the orthogonal of $\mminf$ in $B$. At this stage we observe that for  $f\in B^{\a} \cap B^{\perp}$ \Bk
\begin{equation}\label{p1}
\Lambda^{\a}_{x}f\in B^{\perp}, \qquad \Lambda^{\a}_{v}f \in B^{\perp}
\end{equation}
and that for all $f\in B^{1}$
\begin{equation}\label{p2}
 \partial_{v} f\in B^{\perp}. \Bk
\end{equation}
For~\eqref{p1} we use that the operator $\Lambda_{x}$ is  self-adjoint: $\<\Lambda_{x}^{\a}f,\mm\>=\<f,\Lambda_{x}^{\a} \mm\>=0$ since  ${\Lambda_{x} \mm=\mm}$. \Bk The same proof holds for    $\Lambda_{v}$. The justification of~\eqref{p2} is similar using that $\partial^{*}_{v}=-(v+\partial_{v})$ and\;$(v+\partial_{v})\mm=0$. \medskip

A careful analysis shows that we have in fact the following better results when we restrict to~$B^{\perp}$.
\begin{prop} \ph  \label{decay}
For all $\a, \b \in [0,1]$  there exist $C_{\a},C_{\b}>0$ so that  for all $t > 0$
  \begin{equation}\label{dx}
\|\L^{\a}_{x} \e^{-tK}\|_{B^{\perp}\to B^{\perp}}\leq C_{\a}(1+t^{-3\a/2})e^{- \kappa t}
,\quad  \| \e^{-tK}\L^{\a}_{x}\|_{B^{\perp}\to B^{\perp}}\leq C_{\a}(1+t^{-3\a/2})e^{- \kappa t}
\end{equation}
and
\begin{equation}\label{dv}
\|\L^{\b}_{v} \e^{-tK}\|_{B^{\perp}\to B^{\perp}}\leq C_{\b}(1+t^{-\b/2})e^{- \kappa t},\quad  \| \e^{-tK}\L^{\b}_{v}\|_{B^{\perp}\to B^{\perp}}\leq C_{\b}(1+t^{-\b/2})e^{- \kappa t}.
\end{equation}
In the previous bounds, the constants  $C_{\a}$  and $C_{\b}$ only  depend
 on a finite number of derivatives of~$V$.
\end{prop}

\begin{proof}

For $0 \leq t \leq 1$,  this is a \Bk direct consequence of the preceding proof and the fact that $B^{\perp}$ is stable by  $X_0$, $\Lambda_x^2$ and $\Lambda_v^2$ and therefore $\Lambda^2$, $K$ and $K^*$ by direct computations. 
For $t\geq 1$, the proposition
 is a consequence of the regularizing properties of $e^{-tK}$ stated   in~\cite[Theorem\;0.1]{HN04} \Bk and the spectral gap for $K$: it is proven there that for all $s \in \R $, there exist $N_s >0$ and $C_s >0$ such that
$$
\forall t>0, \qquad \norm{\Lambda^s e^{-tK} \Lambda^s}_{B^{\perp}\to B^{\perp}} \leq C_s (t^s+ t^{-s})e^{-\kappa t}.
$$
Using this and possibly replacing $\kappa$ by $\kappa/2$ gives the result for $t\geq 1$. This completes the proof.
\end{proof}

\begin{rema} \ph \label{divide} In fact possibly replacing once more $\kappa$ by $\kappa/2$, we also get directly that Proposition~\ref{decay} is also true with $K$ replaced by $K/2$. We shall use this just below.
\end{rema}

Similarly to Corollary~\ref{coro23} we have the following

\begin{coro} \ph \label{coro25}
Let $\a, \b \in [0,1]$. Then
\begin{equation}\label{prod1}
 \| \Lambda_{x}^{\a}\e^{-(t-s)K} \Lambda^{1-\b}_{v} \|_{B^{\perp}\to B^{\perp}} \leq  C\big((t-s)^{-1/2+\b/2-3\a/2}+1\big)e^{-\kappa (t-s)},
\end{equation}
and
\begin{equation}\label{prod2}
 \| \Lambda^{\b}_{v}\e^{-(t-s)K} \Lambda^{1-\b}_{v} \|_{B^{\perp}\to B^{\perp}} \leq  C\big((t-s)^{-1/2}+1\big)e^{-\kappa (t-s)}.
\end{equation}
\end{coro}

\begin{prop} \label{gammaplusmoins}
   There exists $C>0$  so that  for  all $\gamma \in [0,2]$  and all $t \geq 0$
\begin{equation}\label{conj}
\|\L^{\gamma} \e^{-tK} \L^{-\gamma} \|_{B^{}\to B^{}}\leq C ,
\end{equation}
and
\begin{equation}\label{conj2}
\|\L^{\gamma} \e^{-tK} \L^{-\gamma} \|_{B^{\perp}\to B^{\perp}}\leq C e^{-\kappa t}.
\end{equation}
In the previous bounds, the constant   only  depends  on a finite number of derivatives of $V$.
\end{prop}

\begin{proof} We only give the proof of~\eqref{conj}, since~\eqref{conj2} can be obtained with the same argument. Recall the definition~\eqref{defB} of the space $B^{\alpha,\beta}_{x,v}$. We first note that it is equivalent
to show that\;$e^{-tK}$ is bounded from $B^{\gamma,\gamma}_{x,v}$ into itself. We first begin with the case $\gamma=2$.  We therefore look, for an initial data $f_0 \in B^{2,2}_{x,v}$ at the equation satisfied by $g= \Lambda^2 f$ in $B$. Let us define the operator
$$X_0 = v.\D_{x} -\D_x V.\D_{v}.$$
 Since
\begin{equation*}
\D_t f + X_0 f -\D_v. (\D_v + v) f = 0, \ \ \ \ \  f_{t= 0} = f_0
\end{equation*}
and from the regularising properties of $e^{-tK}$, we get
\begin{equation*}
\D_t g + X_0 g -\D_v. (\D_v + v) g = [X_0, \Lambda^2]\Lambda^{-2} g, \ \ \ \ \  g_{t= 0} = g_0
\end{equation*}
where we also used that $-\D_v. (\D_v + v)$ and $\Lambda^2$ commute.
Integrating against $g$ in $B$ gives
\begin{equation*}
\D_t \norm{g}^2 \leq \sep{ [X_0, \Lambda^2]\Lambda^{-2}g,g},
\end{equation*}
since $X_0$ is skew adjoint and $-\D_v. (\D_v + v)$ is non-negative. Let us study the right-hand side commutator.
We have
\begin{equation*}
\begin{split}
[X_0, \Lambda_v^2]\Lambda^{-2}
& = [v. \D_x - \D_x V(x) . \D_v, \Lambda_v^2]\Lambda^{-2} \\
& = \sep{ [v,\Lambda_v^2] \D_x - [ \D_v, \Lambda_v^2] \D_x V(x)} \Lambda^{-2}.
\end{split}
\end{equation*}
This gives with a direct computation
$$
\norm{[X_0, \Lambda_v^2]\Lambda^{-2}g}_{B} \leq C \norm{g}_{B}.
$$
We can do exactly the same with $\Lambda_x^2$ (using that $V^{(3)}$ is bounded) and we get on the whole that
$$
\norm{[X_0, \Lambda^2]\Lambda^{-2}g}_{B} \leq 2C \norm{g}_{B}
$$
so that with a new constant $C>0$
\begin{equation*}
\D_t \norm{g}_{B}^2 \leq 2C \norm{g}_{B}^2.
\end{equation*}
We therefore get
$$
\norm{g(t)}_{B} \leq e^{Ct} \norm{g_0}_{B}
$$
which we will use for $t\in [0,1]$.
Using the regularising property of $e^{-tK}$  (\cite[Theorem 0.1]{HN04}), \Bk we also know that for all $t\geq 1$,
$$
\norm{g(t)}_{B} \leq \norm{f(t)}_{B^{2}} \leq C' \norm{f_0}_{B^2} \leq C' \norm{g_0}_{B}.
$$
 Putting \Bk these results together give for all $t\geq 0$,
$$
\norm{g(t)}_{B} \leq C \norm{g_0}_{B}
$$
and  therefore \Bk   $e^{-tK}$ is (uniformly in $t>0$) bounded from $B^2$ to $B^2$.
Now the result is also clear for $\gamma = 0$ by the semi-group property, and by interpolation we get that $e^{-tK}$ is (uniformly in $t>0$) bounded from $B^\gamma$ to $B^\gamma$.
As a conclusion we get
\begin{equation*}
\|\L^{\gamma} \e^{-tK} \L^{-\gamma} \|_{B^{}\to B^{}}\leq C_{\gamma},
\end{equation*}
which was the claim.
\end{proof}

We are now able to state the following interpolation results
 \begin{lemm} \ph\label{leminterpol}
Let $\b \in[0,1]$. Then there exists $C>0$ so that for all $a\in [0,\b]$
\begin{equation*}
\|\L^{\b}_{v} \e^{-tK}  \|_{B^{a}\to B^{ }}\leq C (1+{t^{-(\b-a)/2}}).
\end{equation*}
\end{lemm}

\begin{proof}
For $a=0$, this  follows from Proposition~\ref{prop21}. Next, set $a=\b$, then for $f\in B^{\beta}$
\begin{equation*}
\|\L^{\b}_{v} \e^{-tK} f \|_{ B^{ }}\leq \|\L^{\b} \e^{-tK} f \|_{ B^{ }}\leq \|\L^{\b} \e^{-tK}\L^{-\b}  \|_{ B^{ }\to B}  \|\L^{\b}f  \|_{ B^{ }} \leq C \|f\|_{B^{\b}},
\end{equation*}
by~\eqref{conj}. The general case $a\in [0,\beta]$ is obtained by interpolation.
\end{proof}

 \begin{lemm} \ph\label{leminterpol2}
Let $0\leq a_{0} \leq 2$. Then for all $a_{0}\leq a\leq a_{0}+2$  there exists $C>0$ so that for all $0\leq t\leq1$
\begin{equation*}
\|\L^{a_{0}} \big( \e^{-tK}-1)  \|_{B^{a}\to B^{ }}\leq C {t^{(a-a_{0})/2}}.
\end{equation*}
\end{lemm}

\begin{proof} For $a=a_{0}$, the result follows from~\eqref{conj}. Now we prove the bound for $a=a_{0}+2$, and the general result will follow by interpolation.
We write
\begin{equation*}
\big(1-\e^{-tK}\big)f=\int_{0}^{t}\e^{-s K}Kfds.
\end{equation*}
Then we use that $K:B^{2}\longrightarrow B$ is bounded, and by~\eqref{conj} we get for all $f\in B^{a_{0}+2}$
\begin{eqnarray*}
\|\L^{a_{0}} \big( \e^{-tK}-1)  \|_{B^{a}\to B^{ }}&\leq & \int_{0}^{t} \|\L^{a_{0}} \e^{-sK}\Lambda^{-a_{0}}  \| \|\Lambda^{a_{0}}Kf\|ds\\
&\leq &C t\|f\|_{B^{a_{0}+2}},
\end{eqnarray*}
hence the result.
\end{proof}

We conclude this section with a technical result.

\begin{lemm} \ph
For all $\delta \in \R$  there exists $C_{\delta}>0$ so that
\begin{equation}\label{deriv}
\|\L^{-\delta}_{v}  \Lambda_{v}^{-1}\D_{v} \L^{\delta}_{v}  \|_{B^{}\to B^{\perp}}\leq C_{\delta}.
\end{equation}
In the previous bound, the constant   only  depends  on a finite number of derivatives of $V$.
\end{lemm}

 \begin{proof}
From  to~\cite[Proposition A.7]{HN04}  we directly get that operator 
$ \L^{-\delta}_{v}  \Lambda_{v}^{-1}\D_{v} \L^{\delta}_{v}$ is bounded from $B$ to $B$.  Indeed in the symbolic estimates and  pseudo-differential scales introduced there, the operator
$\D_v$ is of order $1$ with respect to the velocity variable. Now using the stability of $B^\perp$ by $\Lambda_v$ and (\ref{p2}) yield the result.
\end{proof} \Bk

\begin{rema}
We shall see in the next section (Section~\ref{sect36}) that most of the results of this section remain true when $V$ is perturbed by a less regular term $\widetilde{V}\in W^{2,\infty}$. We will need this  for the small time analysis of the equation~\eqref{vpfp}.
\end{rema}

\section{Intermediate results}\label{sect3}

In this section, we gather some intermediate results about the Vlasov-(Poisson)-Fokker-Planck equation.
In the first subsection we state some a priori basic properties
 satisfied by  solutions of the Fokker-Planck equation and then equation~\eqref{vpfp}. In the second one we study more carefully the Poisson term, and in the last one we recall some facts about the equilibrium state.

\subsection{The linear Fokker-Planck equation} \label{basic}

In this section, we just recall from~\cite[Appendix A]{Deg86} some standard and basic results about the behaviour of the solutions of
the following linear Krammers-Fokker-Planck equation
\begin{equation}
  \label{lkfp}
  \left\{
  \begin{aligned}
   &  \D_{t}f+v.\D_{x}f -(\D_x V -v ).\D_{v}f - f - \eps_0 E(t,x) \D_v{ f\Bk} - \Delta_vf=F, \\
    &   f(0,x,v) = f_0(x,v).
    \end{aligned}
    \right.
\end{equation}
Note that equation~\eqref{vpfp} with given field $E$ and $V= \Ve$ enters in this setting  and that the linear Fokker-Planck equation corresponds to $E=0$. In both cases we take $ F= 0$ and point out that we used the commutation estimate $-\D_v (\D_v + v)f = (\D_v + v) (-\D_v)f - f $.

For the following, we take $T >0$ arbitrary and  denote by $X = L^2([0,T]\times \R^d_x, H^1_v(\R^d))$ and consider
the space $Y =\dis  \big\{ f \in X, (\D_t + v.\partial _x - (\D_x V -v). \partial _v) f \in X'\,\big\}$. The following result is classical  and we refer to~\cite[Appendix A]{Deg86} for the proof.

\begin{prop} \ph\label{propos}
Suppose $E \in L^\infty([0,T] \times \R^{d})$, $f_{0} \in L^2(\R^{2d})$ and
$F \in L^2([0,T]\times \R^d_x, H^{-1}_v).$ Then there exists a unique weak solution
$f$ of the equation~\eqref{lkfp} in the class $Y$. Moreover
 \begin{itemize}
  \item[$(i)$] If $f_0 \geq 0$ then $f \geq 0$.
\item[$(ii)$] If  $f_{0}\in L^{\infty}(\R^{2d})$, then for all $0\leq t\leq T$,
\begin{equation*}
\| f(t) \|_{L^{\infty}(\R^{2d})}\leq  e^{ d t} \|f_{0}\|_{L^{\infty}(\R^{2d})}.
\end{equation*}
\end{itemize}
 \end{prop}

This immediately implies the following a priori estimate on the full problem~\eqref{vpfp}.

   \begin{coro}\ph  \label{lemMax1}
Let $ f_{0}\in L^{\infty}(\R^{2d})\cap L^2(\R^{2d})$ be such that $f_{0}\geq 0$ and consider a solution of~\eqref{vpfp} such that the field $E  \in L^\infty([0,T] \times \R^{d})$.
Then, for all $0\leq t\leq T$, $f(t, .) \geq 0$ and
\begin{equation*}
\|f(t)\|_{L^{\infty}(\R^{2d})}\leq \e^{ d t} \|f_{0}\|_{L^{\infty}(\R^{2d})}.
\end{equation*}
 \end{coro}

\subsection{Poisson-Emden equation and equilibrium state}\label{sect32}

The aim of this subsection is to prove that  the potential $\Uinf$ associated to  the
stationary solutions of the Vlasov-Poisson-Fokker-Planck equation is in $W^{\infty,\infty}(\R^{d})$.
Recall that the equation satisfied by $\Uinf$ is
\begin{equation} \label{eqemdengene}
 -\Delta \Uinf =  \frac{e^{-(\Ve+\eps_0 \Uinf)}}{\int
e^{-(\Ve+\eps_{0}\Uinf)} dx}
\end{equation}
where we recall that $\eps_0$ is varying in a small fixed neighbourhood of $0$, and that $\eps_{0}>0$ in the case of dimension $d=2$.

\subsubsection{Case $d=3$}
When we are in the repulsive interaction
case ($\eps_0>0$), the existence and uniqueness of a (Green) solution of this
equation is given by a result of Dolbeault~\cite{Dol91} (see also~\cite{Dol99}) under a
light hypothesis on the external potential. We first quote his
result in dimension $d= 3$ and in the Coulombian case

\begin{prop}[\cite{Dol91}, Section 2] \ph\label{pedol}
Let $U_e \in L^\infty_\textrm{\scriptsize{loc}}(\R^3)$
and $M>0$. Assume that   ${e^{-U_e} \in L^1}(\R^{3})$, then there exists a unique  solution ${U \in L^{3,\infty}}(\R^{3})$ of the
Poisson-Emden equation
\begin{equation} \label{eqemden2}
-\Delta U = M \frac{e^{-(U_e+U)}}{\int e^{-(U_e+U)} dx}.
\end{equation}
Moreover $U\geq 0$.
\end{prop}
The main property of $U$ which will be needed in the following is $U\geq0$, that's why we do not even define precisely the space $L^{3,\infty}(\R^{3})$. For more details, we address to~\cite{Dol91}.\medskip

We then state another result of  Bouchut and Dolbeault  in the Newtonian case ($\eps_0 <0$). This result happens to hold only for small $M$.

\begin{prop}[{\cite[Theorem 3.2 and Proposition 3.4]{BD95}}]  \ph\label{newt}   Assume that $e^{-U_e} \in L^1(\R^3) \cap L^\infty(\R^3)$ and is not identically equal to $0$. Then there exists $M_0 <0$ such that for all $M_0 <M \leq 0$
there exists  a bounded continuous function   of equation~\eqref{eqemden2} such that $\lim_{x \rightarrow \infty} U(x) =0 $.
\end{prop}

Now Assumption~\ref{assumption1} on the exterior potential
$\Ve$ implies that $ e^{-\Ve} \in L^1(\R^3) \cap L^\infty(\R^3) $. As a consequence  we
can apply  Proposition~\ref{pedol} at least in the case when $\eps_0$ is small  to $U=\eps_0 \Uinf$, $U_e=\Ve$, $M=\eps_0$
 and $d=3$ to~\eqref{eqemdengene}  and we get a unique solution $\Uinf$ in
$L^{3,\infty}$ when $\eps_0>0$. Similarly we can apply Proposition~\ref{newt}
 when $\eps_0<0$ and we get $\Uinf \in L^\infty$. Notice that in our context, $|\eps_{0}|$ is small and hence both Propositions~\ref{pedol} and~\ref{newt} apply here.
\bigskip

Actually, the regularity of $\Uinf$ is
improved under the assumption $e^{-\Ve} \in \ss(\R^{3})$, and we can also get some uniformity with respect to the parameter $\eps_0$.

 \begin{prop}  \ph\label{corpeintro} Let $d=3$. Suppose that $\Ve$ satisfies
 Assumption~\ref{assumption1}. Then
the  unique solution $\Uinf$ of the Poisson-Emden equation~\eqref{eqemdengene} is in $W^{\infty,\infty}(\R^{3})$, with semi-norms
uniformly bounded w.r.t. $\eps_0$ varying in a small fixed neighbourhood of $0$.
 \end{prop}

\begin{proof}[Proof of Proposition~\ref{corpeintro}] In order to prove  that
$\Uinf \in W^{\infty,\infty}$, it is sufficient to prove that  the
(Green) solution $\Uinf$ of the following Poisson-Emden-type equation
\begin{equation} \label{eqemdentype}
- \Delta \Uinf =  C_0^{-1} e^{-(\Ve+\eps_0 \Uinf)}
\end{equation}
is in $W^{\infty,\infty}$, where
$$
C_0 = \int e^{-(\Ve(x)+\eps_0 \Uinf(x))} dx
$$
 is the normalization constant. We first work on $\eps_0 \Uinf$ and note that it is given by
$$
 \eps_0 \Uinf = \frac{ \eps_0 C_0^{-1}}{ 4\pi} \frac{1}{|x|} \star e^{-(\Ve +\eps_0 \Uinf)}.
$$
We  then \Bk consider  the Green solution $U_e$ of $-\Delta U_e = e^{-\Ve}$ given by
$$
U_e =   \frac{1}{ 4\pi |x| } \star e^{-\Ve}.
$$
From Propositions~\ref{pedol} and~\ref{newt} we get directly that
    $\eps_0 \Uinf$ exists,  at least \Bk  for $\eps_0$ varying  in a small neighbourhood of $0$, and that  it is either non-negative or uniformly bounded.  It implies that there exists a constant $C>0$ uniform in $\eps_0$  such that
    $0 \leq  \Uinf \leq C U_e$ since we also have
    $$
  \Uinf =   \frac{C_0^{-1}}{4\pi |x|} \star  e^{-(\Ve +\eps_0 \Uinf)}.
$$

From the Hardy-Littlewood Sobolev inequalities or by a direct computation, we have $U_e \in L^p $
for $3 < p \leq  \infty$. Therefore this is also the case for
$\Uinf$. Since we directly have that $-\Delta \Uinf \in L^p$ for  all $p \in [1,\infty]$ from (\ref{eqemdentype}),
we get that
$$
-\Delta \Uinf + \Uinf \in L^p, \ \ \ 3 < p < \infty
$$
and this gives $\Uinf \in W^{2,p}$ by elliptic regularity in $\R^d$
(see for example~\cite{Ste70},~\cite{Won99}).

Now  we \Bk shall use a bootstrap argument to prove that $\Uinf \in
W^{\infty,\infty}$.  Let $ 3 < p <\infty$ be fixed in the
following. We note that
\begin{equation} \label {delta2fois}
\begin{aligned}
 (-\Delta +1)^2 \Uinf & = -\Delta \sep{ C_0^{-1} e^{-(\Ve + \eps_0 \Uinf)}} + 2 C_0^{-1} e^{-(\Ve +\eps_0 \Uinf)} +
 \Uinf
 \end{aligned}
\end{equation}
and we study each term in order to prove that this expression is uniformly in $L^p$.
Since $\Uinf \in
L^\infty$, we have $e^{-\eps_0\Uinf} \in L^\infty$ and  we get for all $ 1\leq i,j\leq 3 \Bk$
$$
\D_{ij} (e^{-\eps_0\Uinf}) = \sep{ -\eps_0  \D_{ij} \Uinf + \eps_0^2 (\D_i \Uinf)(\D_j \Uinf)}e^{-\eps_0 \Uinf} \in
L^{p}
$$
uniformly, since on the one hand  $\Uinf \in W^{2,p}$ uniformly and on the other hand
\begin{equation} \label{ppp}
\forall k, \ \ \D_k \Uinf \in L^{2p} \Longrightarrow  (\D_i \Uinf)(\D_j \Uinf)
\in L^p.
\end{equation}
In a direct way we also get $\D_k e^{-\eps_0 \Uinf} \in L^p$.
Since $e^{-V_e} \in W^{2,p}$ and using the same trick as in~\eqref{ppp}, this  gives from~\eqref{delta2fois} that $\Uinf \in W^{4,p}$
for the arbitrary fixed $
3< p <\infty$. By a bootstrap argument using the same method
we  get that
$$
\Uinf\in W^{2k, p},
$$
for all $k\in \N$ and therefore
$$
\Uinf \in \bigcap_{k\in \N} W^{2k, p} \subset W^{\infty,\infty}.
$$
The
uniformity  w.r.t. $\eps_0$ is also clear and the proof of
  Proposition~\ref{corpeintro} is complete.
  \end{proof}

\subsubsection{Case $d=2$} We consider here only the Coulombian case ($\eps_{0}>0$).

In this context, we are able to prove the following result

 \begin{prop}  \ph\label{corpeintro1} Let $d=2$. Suppose that $\Ve$ satisfies
 Assumption~\ref{assumption1}. Then
the  unique solution $\Uinf$ of the Poisson-Emden equation~\eqref{eqemdengene} is in $W^{\infty,\infty}(\R^{2})$, with semi-norms
uniformly bounded w.r.t. $\eps_0>0$ varying in a small fixed neighbourhood of $0$.
 \end{prop}

\begin{proof}
Notice that when $d=2$, the equation~\eqref{eqemdengene} is equivalent to
$$
  \Uinf =  -\frac1{2\pi}\ln |x|\star  e^{-(\Ve +\eps_0 \Uinf)}.
$$
The existence and uniqueness of a solution $U_{\infty}\in L^{p}(\R^{2})$ for any $1\leq p<\infty$ with $\nabla U_{\infty} \in L^{2}(\R^{2})$ is proved in~\cite[page~199]{Dol91}. Moreover, the maximum principle ensures that $U_{\infty}\geq 0$. It is then straightforward to adapt the proof of the case $d=3$ to conclude.
\end{proof}

In the Newtonian case ($\eps_{0}<0$), and for particular choices of~$V_{e}$ (e.g. $V_{e}(x)=|x|^{2}$, see \cite{BDP}), there exist solutions to the equation~\eqref{eqemdengene}, but uniqueness is unknown, even  under additional assumptions on the solution (radial symmetry, regularity, decay at infinity).  However it would be interesting to prove the trend to equilibrium also in this case. We refer to~\cite{BDP}, where the authors obtained  such a  result for a related problem.

\begin{rema}  \ph To end this section we notice that since $\Uinf \in W^{\infty,\infty}(\R^{d})$, we get that the potential at infinity $\Ve + \eps_0 \Uinf$ satisfies the same hypothesis as $\Ve$ alone. As a consequence it will be possible to apply to $\Kinf$ all the properties obtained for any generic Fokker-Planck operator~$K$ associated to a generic potential $V$ satisfying Assumptions\;\ref{assumption1} and~\ref{assumption2}.  This will be crucial in the next section, in which we study  the exponential convergence to the equilibrium. A second remark is that the total potential at equilibrium is not explicit. In particular, the Green function for the equation $\D_t f + \Kinf f$ is not known. This justifies a posteriori the abstract study (anyway with explicit constants) performed in the linear section. In the next section we first go on with the study of a generic linear Fokker-Planck operator by studying the long time behaviour and the exponential decay in time.
 \end{rema}

\subsection{Uniformity of the spectral gap and heat-operator estimates} \label{spectralgap}

The aim of  this short subsection is to prove that we have indeed a uniform
estimate on the spectral gap for $\Kinf$ with respect to $\eps_0$. Let $d=2$ or $d=3$.
We work with the  operator
\begin{equation*}
\Ke = v.\D_{x} -\D_x \Ve(x).\D_{v}-  \D_{v}.\left(\D_{v}+v\right)
\end{equation*}
and consider a bound from below $\kappa_0$ of the spectral gap of $W$ coming from Assumption~\ref{assumption2}.
 From~\cite[Theorem 0.1]{HN04} \Bk we know that there exist constants $C_0,C>0$   such that for all $t\geq 0$,
\begin{equation} \label{decaybe}
\norm{e^{-t\Ke}}_{\Be^\perp} \leq C_0 e^{-t\kappa_0/C}
\end{equation}
where
$$
\Be = \set{ f \in\ss'(\R^{2d})   \;\;s.t.\;\;  f \mme^{-1/2} \in L^2(\R^{2d})},
$$
and $\mme$ is the Maxwellian associated to $\Ve$  and $\Be^\perp$ is the orthogonal of $\mme$.
We then add to the potential a small perturbation
of type $\eps \Uinf$ with $\Uinf \in W^{\infty, \infty}$.  This will be applied to the  potential \Bk $U_{\infty}$ built   in the preceding subsection. Notice that   $\Uinf \in W^{\infty, \infty}$ with uniform bounds with respect to $0<\epsilon_0\ll 1$.\Bk

The corresponding modified operator is then
\begin{equation} \label{Kinfbis}
\Kinf = v.\D_{x} -\D_x \Vinf(x).\D_{v}-  \D_{v}.\left(\D_{v}+v\right),
\end{equation}
with $\Vinf = \V_e + \eps_0 \Uinf$.
The main result is then the following

\begin{prop}
There exists a small real neighbourhood $\vv$ of $0$ such that for all $t \geq 0$
$$
\norm{e^{-t\Kinf}}_{B^\perp} \leq 4C_0 e^{-t\kappa_0/(8C)}
$$
 uniformly w.r.t. $\eps_0 \in \vv$.
\end{prop}

\begin{proof}
We first  recall that in~\eqref{decaybe} the precise result of~\cite[Theorem 0.1]{HN04}  says that $C_0$ depends on a finite number of semi-norms of $\Ve$ and that
$$
C = \frac{ \min\set{1, \kappa_0}}{64(8 + 3C_e)}
$$
where $C_e = \max \set{ \sup \set{ \Hess(\Ve)^2 - (\frac{1}{4} (\D_x \Ve)^2 -\frac{1}{2} \Delta \Ve) \Id},0}$. Adding  a small perturbation $\eps_0 \Uinf$ with $\Uinf \in W^{\infty, \infty}$ does only change the constant $C$ into $2C$ and $C_0$ into $2 C_0$ and we only have to check that
$\kappa_0$ is changed into $\kappa_0/4$ uniformly in  $\eps_0$ sufficiently small.

For this we look at the spectrum of
$$
W_\infty = -\Delta_x + |\D_x \Vinf|^2/4 - \Delta_x \Vinf /2
$$
and we check that as operators in $L^2 (\R^d)$ we have
\begin{equation*}
\begin{split}
W_\infty & = -\Delta_x + |\D_x \Vinf|^2/4 - \Delta_x \Vinf /2 \\
 & \geq  -\Delta_x + |\D_x \Ve|^2/4 - \Delta_x \Ve /2 +\eps_0^2 |\D_x \Uinf|^2/4 - |\eps_0||\Delta \Uinf|/2 + \eps_0 \D_x\Ve \D_x\Uinf{/2}\\
& \geq   W - \kappa_0/8 + \eps_0 \D_x\Ve \D_x\Uinf{/2}
\end{split}
\end{equation*}
if we take $\eps_0$ sufficiently small so that $\eps_0^2 |\D_x \Uinf|^2/4 + |\eps_0||\Delta \Uinf|/2 \leq \kappa_0/8$. Now there exist constants~$a$ and $b$ such that
$$
|\D_x\Ve \D_x\Uinf| \leq a W +b
$$
since $\Ve$ has its  second order derivatives bounded,
and therefore we get for $\eps_0$ sufficiently small
\begin{equation*}
 W_\infty  \geq   \frac{1}{2} W - \kappa_0/4 .
 \end{equation*}
Since $W \geq \kappa_0$, the minmax principle then directly gives that
$$
W_\infty \geq \kappa_0/4
$$
when restricted to the orthogonal of the $0$-eigenspace.  The proof is complete.
\end{proof}

\begin{rema}
We can also notice that  the natural norm into the weighted spaces
$$
B = \set{ f \in\ss'  \;\;s.t.\;\;   f \mminf^{-1/2} \in L^2} \quad \textrm{ and } \quad
  \Be= \set{ f \in\ss'  \;\;s.t.\;\;   f \mme^{-1/2} \in L^2}
  $$
where $\mm$ is the Maxwellian associated to $\Ve$,   are equivalent with an equivalence constant bounded by~$1/2$ uniformly in $\eps_0$ small enough. This justifies the use of the norms associated to the space $B$ instead of the one associated to $\Be$ in the statement of the main theorems of this article.
\end{rema}

\subsection{Estimates on the Poisson term}

In the following lemma we crucially use the fact that we work in weighted Sobolev spaces instead of flat ones and that $\mminf \in \ss(\R^{2d})$ uniformly in $|\eps_0| \ll 1$, as proven in the preceding subsection. We have

\begin{lemm} \ph \label{HLSSE}
Let $\alpha \in[0,1]$ then there exists $C>0$ such that for all   $h_0 \in B^\alpha$
$$
\Big\Vert \int h_0 dv\Big\Vert_{H^\alpha_x} \leq C \norm{h_0}_{B^\alpha}.
$$
\end{lemm}

\begin{proof}
We work by interpolation. Let us first consider the case $\alpha =0$. By Cauchy-Schwarz,
\begin{eqnarray*}
\Big\Vert \int h_{0}dv\Big\Vert_{L^2_x} &=&\Big\Vert \int h_{0}\mminf^{-1/2}\mminf^{1/2}dv\Big\Vert_{L^2_x} \\
 &\leq& \Big\Vert \Big(\int  h_{0}^{2}\mminf^{-1}dv\Big)^{1/2}\big(\int \mminf dv\big)^{1/2}\Big\Vert_{L^2_x}
  \leq  C_0 \norm{h_0}_B.
\end{eqnarray*}
Now we consider the case $\alpha = 1$. We write
\begin{eqnarray*}
\Big\Vert \partial _x \int h_{0}dv\Big\Vert_{L^2_x} &=& \Big\Vert \int \partial _x  h_{0}dv\Big\Vert_{L^2_x} \\
& \leq & \Big\Vert \int (\partial _x + \D_x V)  h_{0}dv\Big\Vert_{L^2_x} +
\Big\Vert \int (\D_x V)  h_{0}dv\Big\Vert_{L^2_x}  \\
 &=&\Big\Vert \int \big((\partial _x + \D_x V)h_{0}\big) \mminf^{-1/2}\mminf^{1/2}dv\Big\Vert_{L^2_x} +
 \Big\Vert \int h_{0} \mminf^{-1/2} \big( \D_x V\mminf^{1/2} \big) dv\Big\Vert_{L^2_x}  \\
 &\leq& \Big\Vert \Big(\int  \big((\partial _x + \D_x V)h_{0}\big)^{2}\mminf^{-1}dv\Big)^{1/2}\big(\int \mminf dv\big)^{1/2}\Big\Vert_{L^2_x} \\
 && \qquad \qquad \qquad + \Big\Vert \Big(\int  h_{0}^{2} \mminf^{-1}dv\Big)^{1/2}\Big(\int (\D_x V)^2 \mminf dv\Big)^{1/2}\Big\Vert_{L^2_x} \\
  &\leq &  C \Bk \Big\Vert(\partial _x + \D_x V ) h_0\Big\Vert_B + C\norm{ h_0}_B \leq C \norm{ h_0}_{B^1}
\end{eqnarray*}
where we used that $(\D_x V)^2 \mminf  \in L^\infty_x L^2_v$, and that
 \begin{equation*}
\begin{split}
\norm{(\partial _x + \D_x V) h_0}_B^2 + \norm{h_0}^2_B & = \sep{-\partial _x(\partial _x + \D_x V) h_0, h_0}_B +  \norm{h_0}^2_B \\
& = (\Lambda^2_x h_0, h_0)_B = \norm{ \Lambda_{x} h_0}_{B}^2 \leq \norm{h_0}^2_{B^1}. 
\end{split}
\end{equation*}
\Bk
This gives the result for $\alpha=1$. The complete result follows by interpolation.
\end{proof}

  \begin{lemm}\ph\label{lem22}
 Assume that $d=2$ or $d=3$. Let $h_{0}\in B$ and  denote by
 \begin{equation*}
 E_{0}(x)=\frac{x}{|x|^{d}}\star \int h_{0}(x,v)dv.
 \end{equation*}
 \begin{enumerate}[(i)]
  \item Case $d=2$.
  For all   $ 0<\eps \leq 1/2$ there exists $C>0$ so that
 \begin{equation}\label{claim2}
\|E_{0}\|_{L^{\infty}(\R^{2})} \leq C \| h_{0}\|_{B^{\eps}}.
\end{equation}
 \item Case $d=3$.  For all     $ 0<\eps \leq 1/2$ there exists $C>0$ so that
\begin{equation}\label{claim}
\|E_{0}\|_{L^{\infty}(\R^{3})} \leq C \| h_{0}\|_{B^{1/2+\eps}}.
\end{equation}
 \end{enumerate}
\end{lemm}

\begin{proof} Let us first recall the Hardy-Littlewood-Sobolev inequality (see e.g.~\cite{Ler14})
which will   be \Bk useful in the sequel. For all $1<p,q<+\infty$ such that $\frac1q-\frac1p+\frac1d=0$
\begin{equation}\label{HLS}
\big\Vert \frac{x}{|x|^{d}}\star f\big\Vert_{L^q(\R^d)}\leq C \Vert  f\Vert_{L^p(\R^d)}.
\end{equation}

We  prove $(ii)$. We consider the Fourier multiplier $L_x = (1-\Delta_x)^{1/2}$.
Then,  by Hardy-Littlewood-Sobolev and the Sobolev embeddings, for any $\eps>0$
\begin{equation*}
\|E_{0}\|_{L^{\infty}(\R^{3})} \leq C \big\|L^{\eps}_{x}\int h_{0}dv\big\|_{L^{3}(\R^{3})}\leq C \big\|L^{1/2+\eps}_{x}\int h_{0}dv\big\|_{L^{2}(\R^{3})} \leq \big\|\int h_{0}dv\big\|_{H^{1/2+\eps}(\R^{3})}.
\end{equation*}
Using Lemma~\ref{HLSSE} with $\alpha = 1/2+\eps$ we get
~\eqref{claim}.

The proof  of $(i)$ is analogous with $L_{x}^{\eps}$ replaced with $L_{x}^{1/2+\eps}$.
\end{proof}

  \begin{coro}\ph\label{lem23}
 Assume that $d=2$ or $d=3$. Let $f_{0}\in B^{\perp}$ and  denote by
 \begin{equation*}
 E_{0}(t,x)=\frac{x}{|x|^{d}}\star \int\e^{-tK}f_{0}(x,v)dv.
 \end{equation*}
  \begin{enumerate}[(i)]
  \item Case $d=2$. For all   $ 0<\eps \leq 1/2$ there exists $C>0$ so that for all $t>0$
 \begin{equation} \label{e00}
 \|E_{0}(t,.)\|_{L^{\infty}(\R^{2})}\leq C (1+t^{-3\eps/2})e^{- \kappa   t}\|f_{0}\|_{B}.
 \end{equation}
   \item Case $d=3$.  For all   $ 0<\eps \leq 1/2$ there exists $C>0$ so that for all $t>0$
 \begin{equation} \label{e000}
 \|E_{0}(t,.)\|_{L^{\infty}(\R^{3})}\leq C e^{- \kappa   t}\|f_{0}\|_{B^{1/2+\eps}}.
 \end{equation}
   \end{enumerate}
\end{coro}

\begin{proof} $(i)$. We apply the result of Lemma~\ref{lem22} to the  case $h_{0}=\e^{-tK}f_{0}$ for some $f_{0}\in B^{\perp}$, then
\begin{equation}\label{star}
\|E_{0}(t,.)\|_{L^{\infty}(\R^{2})} \leq C\|\L^{\eps}_{x} \e^{-tK} f_{0}\|_{B^{}}.
\end{equation}
Thus  estimate~\eqref{dx} together with~\eqref{star} implies
\begin{equation*}
\|E_{0}(t,.)\|_{L^{\infty}(\R^{2})} \leq C (1+t^{-3\eps/2})e^{-\kappa t}\|f_{0}\|_{B^{}},
\end{equation*}
which was to prove.

 $(ii)$. By~\eqref{claim} and~\eqref{conj}, we obtain
\begin{equation*}
\|E_{0}(t,.)\|_{L^{\infty}(\R^{3})} \leq C\|\L^{1/2+\eps} \e^{-tK} \L^{-1/2-\eps} \|_{B^{\perp}\to B^{\perp}}\| \L^{1/2+\eps} f_{0}\|_{B^{}}\leq Ce^{-\kappa t} \|   f_{0}\|_{B^{1/2+\eps}},
\end{equation*}
which was the claim.
\end{proof}

\subsection{Integral estimates} In this subsection we give  a technical result.
 \begin{lemm} \ph Let $\gamma_{1},\gamma_{2}, c >0$ and assume that $\gamma_{1}\leq 1$. Then there exists $C>0$ so that for all~$t>0$
 \begin{equation} \label{Int}
 \int_{0}^{t}\big(s^{-1+\gamma_{1}}+1\big)\big((t-s)^{-1+\gamma_{2}}+1\big)\e^{-c (t-s)}ds\leq  \left\{
  \begin{aligned}
   & C\big(t^{-1+\gamma_{1}+\gamma_{2}}+1\big)\quad& \text{for}\quad t\leq 1, \\
   & C\quad &\text{for}\quad t\geq 1.
    \end{aligned}
    \right.
\end{equation}
\end{lemm}

\begin{proof} The proof is elementary: we  expand \Bk the r.h.s. of~\eqref{Int} and estimate each piece. Let $t\leq 1$, then
\begin{equation*}
 \int_{0}^{t}\big(s^{-1+\gamma_{1}}+1\big)\big((t-s)^{-1+\gamma_{2}}+1\big)\e^{-c (t-s)}ds \leq  \int_{0}^{t}\big(s^{-1+\gamma_{1}}+1\big)\big((t-s)^{-1+\gamma_{2}}+1\big) ds,
\end{equation*}
Firstly,
\begin{equation*}
 \int_{0}^{t} s^{-1+\gamma_{1}} (t-s)^{-1+\gamma_{2}} ds=C_{\gamma_{1},\gamma_{2}}t^{-1+\gamma_{1}+\gamma_{2}},
\end{equation*}
by a simple change of variables. Then for $t\leq 1$
\begin{equation*}
\int_{0}^{t} s^{-1+\gamma_{1}} ds +\int_{0}^{t} (t-s)^{-1+\gamma_{2}} ds \leq C,
\end{equation*}
and this yields the result. Now we assume that $t\ge1$. Then on the one hand
 \begin{eqnarray*}
 \int_{0}^{1}\big(s^{-1+\gamma_{1}}+1\big)\big((t-s)^{-1+\gamma_{2}}+1\big)\e^{-c (t-s)}ds &\leq&  C \e^{-c t}\int_{0}^{1}\big(s^{-1+\gamma_{1}}+1\big)\big((t-s)^{-1+\gamma_{2}}+1\big) ds\\
  &\leq&  C,
 \end{eqnarray*}
and on the other hand, since $\gamma_{1}\leq1$
 \begin{eqnarray*}
 \int_{1}^{t}\big(s^{-1+\gamma_{1}}+1\big)\big((t-s)^{-1+\gamma_{2}}+1\big)\e^{-c (t-s)}ds &\leq&  C \int_{0}^{t} \big((t-s)^{-1+\gamma_{2}}+1\big)\e^{-c (t-s)} ds\\
  &\leq&  C,
 \end{eqnarray*}
which completes the proof.
\end{proof}


\subsection{Low regularity heat estimates}\label{sect36}

In this subsection we show how some of the previous results on the Fokker-Planck operator with potential satisfying Assumption~\ref{assumption1} remain valid when the potential
is of type
$$
V = \Ve +\eps_0  U_0
$$
where $\Ve$ satisfies Assumption~\ref{assumption1},  $U_0 \in W^{2,\infty}$ and  $\vert \eps_0\vert\leq 1 $.  This will be applied in Section~\ref{sect5} when the study for short time will be done.

In the following we denote by
\begin{equation*}
K = v.\D_{x} -\D_x \Ve(x).\D_{v}-  \D_{v}.\left(\D_{v}+v\right)
\end{equation*}
and
$$
K_0 = K -  \eps_0  \D_x U_0 (x)\Bk.\D_v.
$$
Note that the Hilbert spaces of type $B$ defined in~\eqref{defb2}  with either $\mminf$ (defined in~\eqref{defmme}) or $\mme$ (when $\Ve +\eps_0 \Uinf$ is replaced there by $\Ve$ only) or even $\mm_0$ (when $\Ve +\eps_0 \Uinf$   is replaced there by $\Ve +\eps_0 U_0$)  are all equal with equivalent norms uniformly in $0 \leq \eps_0 \leq 1$ and depending only on the norm sup of $U_0$ or $\Uinf$.\medskip

 We will need the following result

\begin{lemm}\ph The domains of $K$ and $K_0$ coincide, they are both maximal accretive with $\mm^{1/2} \ss$ as a core.
\end{lemm}

\begin{proof}
This is clear for $K$ as already noticed and used (see~\cite{HN04}). The difficulty is that $K_0$ has only $W^{1, \infty}$ coefficients.  There exists $C_{0}>0$ such that  $\|\partial_{x}U_{0}\|_{L^{\infty}}\leq C_{0}$, and then for any\;$\eta>0$, there exists $C_{\eta}>0$ such that
$$
\norm{ \D_x U_0.\D_v f }_{B} \leq  C_0 \norm{\D_v f }_{B} \leq  \eta\norm{K f}_{B} + C_\eta \norm{f}_{B},
$$
which  directly implies that the domains are the same, see e.g. \cite[ Chapter III, Lemma 2.4]{EN00}.   The fact that $\mm^{1/2} \ss$ is a core is also a direct consequence of this inequality.
\end{proof}

We now prove that some results from Section~\ref{sect2} about semigroup estimates remain true
for the  new operator $K_0$ with non-smooth coefficients.\medskip

We begin with a general Proposition

\begin{prop}\ph  \label{lowregularity}
Let us consider the operator $K_0$ with potential $\Ve + \eps_0 U_0$. Then there exists $C_0>0$ such that the  following is true uniformly in $\eps_0 \in [0,1]$ and $t\in (0,1]$
\begin{enumerate}[(i)]
\item $\forall \gamma \in [0,1]$, \qquad$\norm{ \Lambda^\gamma e^{-tK_0} \Lambda^{-\gamma}}_{B\to B} \leq C_0$,\vspace{5pt}
\item $\forall \beta \in [0,1]$, \qquad $\norm{ \Lambda_v^\beta e^{-tK_0} }_{B\to B} \leq C_0 t^{-\beta/2}$,\vspace{5pt}
\item $\forall \alpha \in [0,1]$, \qquad $\norm{ \Lambda_x^\alpha e^{-tK_0} }_{B\to B} \leq C_0 t^{-3\alpha/2}$,\vspace{5pt}
\item $\forall a   \in [0,2]$ and $f \in B^{a}$, \qquad $\norm{ (e^{-tK_0}-1)f}_{B} \leq C_0t^a \norm{f}_{B^a}.$
    \end{enumerate}
    \end{prop}

\begin{proof}
We first note that the proof of point $(iv)$ given in Lemma~\ref{leminterpol2} is unchanged (for $a_0 = 0$) under the new assumptions on the potential $V$, and uniformly w.r.t. $\eps_0$. For points $(iii)$ and $(ii)$ this is the same w.r.t. the proof of Proposition~\ref{prop21} and we emphasise that the constants  only depend on the second derivatives of the potential, which are here uniformly bounded w.r.t. $\eps_0$.

\bigskip It therefore only remains to check point $(i)$  for which the proof of point~\eqref{conj} cannot be directly adapted, since we have to restrict here to the case when $\gamma \in [0,1]$. We have to show that $e^{-tK}$ is bounded from $B^{\gamma,\gamma}_{x,v}$ into itself. We first begin with the case $\gamma=1$. We  now \Bk use  that
 $$
 \norm{f}_{B^1} \sim \norm{ \Lambda f}_B \sim \norm{ (\D_x + \D_x V) f }_{B} + \norm{ (\D_v + v) f }_{B}
 $$
 with uniform w.r.t. $\eps_0$  equivalence constants, since $U_0 \in W^{2, \infty}$.
 We therefore look, for an initial data $f_0 \in B^{1,1}_{x,v}$ at the equation satisfied by $g= (\D_x + \D_x V) f$ and $h=  (\D_v + v) f$ in $B$. We consider again  the \Bk operator
$X_0 = v.\D_{x} -\D_x \Ve.\D_{v}$.
 Since
\begin{equation*}
\D_t f + X_0 f  -\eps_0 \D_x U_0.\D_{v}f - \D_v. (\D_v + v) f = 0, \ \ \ \ \  f_{t= 0} = f_0
\end{equation*}
we get the system
\begin{equation*}
\begin{split}
& \D_t g + X_0 g - \eps_0 \D_x U_0 \D_v g -\D_v. (\D_v + v) g =   \Hess V h   \\
& \D_t h + X_0 h - \eps_0 \D_x U_0 \D_v  h\Bk  -\D_v. (\D_v + v)  h \Bk =   -h - g+\eps_0 \partial_x U_0 f , \\
 & \textrm{ with }\quad \qquad g_{t= 0} = g_0\in B\quad  \textrm{ and }\quad  h_{t= 0} = h_0 \in B.
  \end{split}
\end{equation*}
 Integrating the three last equations against respectively $f$, $g$ and  $h$ in $B$   gives,
\begin{equation*}
\D_t ( \norm{f}_{B}^2+ \norm{g}_{B}^2 + \norm{h}_{B}^2)  \leq C ( \norm{f}_{B}^2+ \norm{g}_{B}^2 + \norm{h}_{B}^2)
\end{equation*}
since $V$ has a Hessian uniformly bounded w.r.t. $\eps_0$.
We therefore get
$$
 \norm{f(t)}_{B} + \norm{g(t)}_{B} + \norm{h(t)}_{B} \leq C_1e^{C_2 t} (\norm{f_0}_{B}+ \norm{g_0}_{B} + \norm{h_0}_{B})
$$ \Bk
and we get that $e^{-tK_0}$ is (uniformly in $t \in [0,1]$ and $\eps_0 \in [0,1]$) bounded from $B^1$ to $B^1$.
Now the result is also clear for $\gamma = 0$ by the semi group property, and by interpolation we get that $e^{-tK_0}$ is (uniformly in $t \in [0,1]$ and $\eps_0 \in [0,1]$) bounded from $B^\gamma$ to $B^\gamma$ for $\gamma \in [0,1]$.
As a conclusion we get
\begin{equation*}
\|\L^{\gamma} \e^{-tK} \L^{-\gamma} \|_{B^{}\to B^{}}\leq C_{\gamma}.
\end{equation*}
This concludes the proof of point  $(i)$ and the proof of the  Proposition. \Bk
\end{proof}

As a consequence, a certain number of  results  of Section~\ref{sect2} remain true with proofs without changes. \Bk
We gather them in the following corollary.

\begin{coro} \ph \label{corlow}
 There exists $C>0$ such that the  following is true uniformly in $\eps_0 \in [0,1]$ and $t\in (0,1)$
\begin{enumerate}[(i)]
\item $\forall \beta \in [0,1]$, \; $\forall a \in [0,\beta]$,\qquad $\norm{ \Lambda_v^\beta e^{-tK_0} }_{B^a \rightarrow B} \leq C(1+ t^{-(\beta-a)/2})$,\vspace{5pt}
\item $\forall \beta \in [0,1]$,  \qquad $\norm{ \Lambda_v^\beta e^{-tK_0} \Lambda_v^{1-\beta} }_{B\to B} \leq C(1+ t^{-1/2})$,\vspace{5pt}
\item $\forall \alpha, \beta \in [0,1]$,   $\norm{ \Lambda^\alpha e^{-tK_0} \Lambda_v^{1-\beta}}_{B\to B} \leq C(1+ t^{-1/2+\beta/2 -3\alpha/2})$.
    \end{enumerate}
    \end{coro}

\begin{proof}
The proof of $(i)$  follows the one of Lemma~\ref{leminterpol} thanks to points $(i)$, $(ii)$ and $(iii)$  in  Proposition~\ref{lowregularity}. Points  $(ii)$ and $(iii)$ are  consequences respectively of $(ii)$  and $(iii)$  of Proposition~\ref{lowregularity} since
$$
\norm{ \Lambda_v^\beta e^{-tK_0} \Lambda_v^{1-\beta}}_{B\to B} \leq \norm{ \Lambda_v^\beta e^{-tK_0/2} }_{B\to B} \norm{ e^{-tK_0/2} \Lambda_v^{1-\beta} }_{B\to B}
$$
and
$$
\norm{ \Lambda^\alpha e^{-tK_0} \Lambda_v^{1-\beta} }_{B\to B} \leq \norm{ \Lambda^\alpha e^{-tK_0/2} }_{B\to B} \norm{ e^{-tK_0/2} \Lambda_v^{1-\beta} }_{B\to B}.
$$
  \end{proof}

\begin{rema} \ph\label{rema1}
Let us observe that if one only has $f_{0}\in B(\R^{6})\cap L^{\infty}(\R^{6})$,  one can prove that~$U_0$ defined in~\eqref{hyp} satisfies $U_{0}\in W^{2,p}(\R^{3})$ for any $2\leq p<\infty$.  In other words, the assumption  $U_{0}\in W^{2,\infty}(\R^{3})$  fills in an $\eps-$gap of regularity.  More precisely, \Bk let $p\geq 2$ and $q\leq 2$ such that $1/p+1/q=1$. Then, by H\"older
\begin{equation*}
|\Delta U_{0}|=\int f_{0} dv \leq \big(\int f_0^{p}\M^{-1} dv \big)^{1/p}   \big(\int \M^{q/p} dv \big)^{1/q} .
\end{equation*}
Thus using that $    \dis \int  \M^{q/p} dv \in L^{\infty}(\R^{3}) $, we get
\begin{equation*}
\int |\Delta U_{0}|^{p}dx  \leq C \int f_0^{p}\M^{-1} dv dx \leq C \|f_0\|^{p-2}_{L^{\infty}(\R^{6})}\|f_0\|^{2}_{B},
\end{equation*}
which implies that $U_{0}\in W^{2,p}(\R^{3})$ by elliptic regularity.
\end{rema}

Now we prove a result that will be useful for the short time analysis in the next section. Again we work with the linear Fokker-Planck operator $K_0$ with potential $V_{e} + \eps_0 U_0$.

 \begin{lemm}\ph  \label{convolmoins}
Assume that  $d=3$ and $a>1/2$. Let $f_{0}\in B^{a}(\R^{6})\cap L^{\infty}(\R^{6})$ and  denote by
 \begin{equation*}
 S_{0}(t,x)=\frac{x}{|x|^{3}}\star \int\big(\e^{-tK_{0}}-1\big)f_{0}(x,v)dv.
 \end{equation*}
 Then for all $\eps\ll 1$ and $0\leq t\leq 1$ and uniformly in $\eps_0 \in [0,1]$ we have
 \begin{equation}\label{211}
 \|S_{0}(t)\|_{L^{\infty}(\R^{3})}\leq C t^{a/3-\eps}\big(\|f_{0}\|_{L^{\infty}}+\|f_{0}\|_{B^{a}}\big).
 \end{equation}
 \end{lemm}

 \begin{proof}
In the sequel, $0\leq t\leq 1$ is fixed.  Let $\s=a-1/2>0$ and let $q>3/\s$ be large.    Then by the  Gagliardo-Nirenberg inequality
 \begin{equation}\label{GN}
 \|S_{0}\|_{L_{x}^{\infty}}\leq C \|S_{0}\|^{1-\frac{3}{\s q}}_{L_{x}^{q}} \|S_{0}\|^{\frac{3}{\s q}}_{W_{x}^{\s,q}},
 \end{equation}
 and we now estimate the previous terms.

By~\eqref{HLS}, there exists $p<3$ (with $p\longrightarrow 3$ when $q\longrightarrow +\infty$) such that
  \begin{equation*}
  \|S_{0}\|_{L_{x}^{q}} \leq C \big\| \int h_{0}dv\big\|_{L_{x}^{p}},
  \end{equation*}
where $h_{0}=\big(\e^{-tK_{0}}-1\big)f_{0}$. Then, by H\"older (where $p'$ is the conjugate of $p$)
\begin{eqnarray*}
 \int |h_{0}|dv&=&\int \big( |h_{0}| \mminf^{-1/p}\big)\mminf^{1/p}dv\\
 &\leq&\Big(\int | h_{0}|^{p}\mminf^{-1}dv\Big)^{1/p}\Big(\int \mminf ^{p'/p}dv\Big)^{1/p'}\\
  &\leq&C\Big(\int | h_{0}|^{p}\mminf^{-1}dv\Big)^{1/p}.
\end{eqnarray*}
This implies that
 \begin{equation}\label{S1}
  \|S_{0}\|_{L_{x}^{q}} \leq C \big\| \int h_{0}dv\big\|_{L_{x}^{p}} \leq C\Big(\int | h_{0}|^{p}\mminf^{-1}dvdx\Big)^{1/p}\leq  C \|h_{0}\|^{1-2/p}_{L^{\infty}}  \|h_{0}\|^{2/p}_{B} .
  \end{equation}

Now, by point $(iv)$ of Proposition~\ref{lowregularity} we have $ \|h_{0}\|_{L^{\infty}} \leq C \|f_{0}\|_{L^{\infty}} $, and by Lemma~\ref{leminterpol2}, $ \|h_{0}\|_{B} \leq Ct^{a/2} \|f_{0}\|_{B^{a}} $, hence
\begin{equation*}
 \|S_{0}\|_{L_{x}^{q}} \leq Ct^{a/p }\big(\|f_{0}\|_{L^{\infty}}+\|f_{0}\|_{B^{a}}\big).
\end{equation*}
Next, by~\eqref{HLS} and Sobolev (recall that $p\sim 3$ for $q$ large)
\begin{equation*}
  \|S_{0}\|_{W_{x}^{\s,q}} \leq C \big\| (1-\Delta_x)^{\s/2}\int h_{0}dv\big\|_{L_{x}^{p}}  \leq C \big\| (1-\Delta_x)^{a/2}\int h_{0}dv\big\|_{L_{x}^{2}} ,
\end{equation*}
since $\s+1/2=a$. Then we proceed as in the proof of~\eqref{claim} to get
\begin{equation}\label{S2}
  \|S_{0}\|_{W_{x}^{\s,q}} \leq C \|h_{0}\|_{B^{a}}\leq C  \|f_{0}\|_{B^{a}}.
\end{equation}
Fix $\eps\ll1$. Then for $q\gg1$, we combine~\eqref{GN},~\eqref{S1} and~\eqref{S2} to get~\eqref{211}.
 \end{proof}

\Bk

\section{\texorpdfstring{Proof of Theorem~\ref{thm1} (case $\boldsymbol{d=2}$)}{Proof of Theorem 1.2 (case d=2)}}\label{sect4}

\subsection{Functional setting}\label{Sub41}
To begin with, we introduce the functional framework which will be used in both cases  $d=2$ or $d=3$.

To show the trend to equilibrium,  we look for a solution of the form $f=f_{\infty}+g$ with $f_{\infty}=c\mminf$ and $g\in B^{\perp}$. The normalization $\int fdxdv=\int \mminf dxdv=1$ then implies that $f_{\infty}=\mminf$. Hence we write
\begin{equation*}
f=\M_{\infty}+g,\qquad E=E_{\infty}+ F,
\end{equation*}
with
\begin{equation*}
 \partial_x\Uinf=E_{\infty}=-  \frac{1}{|\S^{d-1}|}\frac{x}{|x|^{d}}\star \int \M_{\infty}dv,\qquad F=-  \frac{1}{|\S^{d-1}|}\frac{x}{|x|^{d}}\star \int g dv.
\end{equation*}
In the sequel  denote by
\begin{equation*}
K=\Kinf.
\end{equation*}
We want to take profit of the regularization property stated in Lemma~\ref{lem22}, thus we look for a solution of the form
\begin{equation*}
g=\e^{-tK}g_{0}+h,\qquad F=F_{0}+G,
\end{equation*}
with
\begin{equation*}
 F_{0}=-  \frac{1}{|\S^{d-1}|}\frac{x}{|x|^{d}}\star \int\e^{-tK}g_{0}dv,\qquad G=-  \frac{1}{|\S^{d-1}|}\frac{x}{|x|^{d}}\star \int h dv,
\end{equation*}
and $h(0)=G(0)=0$. At this stage we observe that $f_{0}=\mminf +g_{0}$ and that for all $t\geq 0$, $\e^{-tK}f_{0}=\mminf +\e^{-tK}g_{0}$.\medskip

We construct the solution with a  fixed point argument on $(h,G)$, and therefore we define the map   $\Phi=(\Phi_{1},\Phi_{2})$ given by
\begin{equation*}
\Phi_{1}(h,G)(t)= \eps_{0}\int_0^t \e^{-(t-s)K}  \big(F_{0}(s)+G(s)\big)\D_v \big(  \mminf+\e^{-sK}g_{0}+h(s)    \big)ds
\end{equation*}
\begin{equation*}
\Phi_{2}(h,G)(t)=-  \frac{\eps_0}{|\S^{d-1}|} \frac{x}{|x|^{d}}\star \int_{\R^{d}}\int_0^t \e^{-(t-s)K}  \big(F_{0}(s)+G(s)\big)\D_v \big(  \mminf+\e^{-sK}g_{0}+h(s)    \big)dsdv,
\end{equation*}
and we observe that $(f,E)$ solves~\eqref{int} if and only if $(h,G)=\Phi(h,G)$.  For $\a,\b,\gamma,\delta,\s  \geq 0 \Bk$ define the norms
\begin{equation*}
\|h\|_{X_{\delta}^{\a,\b}}=\sup_{t\geq 0} \Big(\frac{ t^{\delta}}{1+t^{\delta}}e^{\s\kappa t}\|h(t,.)\|_{B_{x,v}^{\a,\b}(\R^{2d})}\Big),
\end{equation*}
\begin{equation*}
\|G\|_{Y_{ }}=\sup_{t\geq 0} \Big( e^{\s\kappa t}\|G(t,.)\|_{L^{\infty}(\R^{d})}\Big),
\end{equation*}
define  the Banach space
$$Z:=X_{\delta}^{\a,\b}\times Y_{} ,\qquad \text{with}\qquad \|(h,G)\|_{Z}=\max\big(\|  h\|_{X_{\delta}^{\a,\b}},  \|G\|_{Y_{ }}\big),$$
and   denote by $\Gamma_{1}$ its unit ball. In each of the cases $d=2$ or 3, for a given initial condition $g_{0}$, we will prove that if $|\eps_{0}|<1$ is small enough, the map $\Phi$ is a contraction of the  ball  $\Gamma_{1} \subset Z$. To alleviate notations, we assume in the sequel that $\eps_{0}>0$.


\subsection{The fixed point argument in the case ${\boldsymbol {d=2}}$} This case is the easiest. Let $g_0\in B$. We can fix here $\a=\b =\delta=0$. Let $\eps\ll 1$ and fix  $\s=1/2$. For simplicity, we write $X=X^{0,0}_0$.

We proceed in two steps. Recall that $\Gamma_{1}$ is the unit ball of $Z$. Then  \medskip

{\bf Step1: $\Phi$ maps the ball $\Gamma_{1} \subset Z$ into itself}\medskip

$\bullet$ We estimate $ \Phi_{1}(h,G)$ in $X$.
By~\eqref{deriv} and~\eqref{p2}, we have for all $t\geq 0$
\begin{eqnarray}
\|   \Phi_{1}(h,G)(t)\|_{B^{}}&\leq & \eps_{0}\int_{0}^{t}   \|  \e^{-(t-s)K}  \big(F_{0}+G\big)\D_v \big(\mminf+\e^{-sK}g_{0}+h(s)\big)\|_{B^{}}ds \nonumber\\
&\leq & C\eps_{0} \int_{0}^{t}  \|F_{0}+G\|_{L^{\infty}(\R^{2})} \|  \e^{-(t-s)K} \Lambda^{}_{v} \|_{B^{\perp}}\|  \mminf+\e^{-sK}g_{0}+h(s) \|_{B^{}}ds,\label{2140}
\end{eqnarray}
and we estimate each factor in the previous integral.\medskip

\underline{Estimation of $\|  \mminf+\e^{-sK}g_{0}+h(s)\|_{B^{}}$:} We use that  $\mminf \in B$, and by~\eqref{maxi} we obtain
\begin{eqnarray}
\|  \mminf+\e^{-sK}g_{0}+h(s) \|_{B^{}} &\leq & \|  \mminf \|_{B^{}} +\|\e^{-sK}g_{0}\|_{B^{ }}+\|h(s) \|_{B^{ }}\nonumber \\
&\leq & C (1+\|h\|_{X}).\label{370}
\end{eqnarray}

\underline{Estimation of $ \|F_{0}+G\|_{L^{\infty}(\R^{3})}$:}
By~\eqref{e00} we get
\begin{eqnarray}\label{380}
  \|F_{0}+G\|_{L^{\infty}(\R^{2})}&\leq&   \|F_{0}\|_{L^{\infty}(\R^{2})} + \|G\|_{L^{\infty}(\R^{2})}\nonumber \\
  &\leq&   C (1+s^{ -3\eps/2})e^{-\s \kappa s}\|g_{0}\|_{B^{}} + C\e^{-\s\kappa s}\|G\|_{Y}\nonumber\\
&\leq &    C (1+s^{-3\eps/2})\e^{-\s \kappa s}(1+\|G\|_{Y}).
 \end{eqnarray}

 \underline{Estimation of $  \|  \e^{-(t-s)K} \Lambda_{v} \|_{B^{\perp}\to B^{\perp}}$:}
This follows from~\eqref{dv}
\begin{equation}\label{2150}
 \| \e^{-(t-s)K} \Lambda_{v} \|_{B^{\perp}\to B^{\perp}}\leq  C\big((t-s)^{-1/2}+1\big)e^{-\kappa(t-s)}.
\end{equation}
Therefore by~\eqref{2140},~\eqref{370},~\eqref{380} and~\eqref{2150} we have
 \begin{equation*}
 \| \Phi_{1}(h,G)(t)\|_{B^{}} \leq C \eps_{0}(1+\|h\|_{X})(1+\|G\|_{Y})    \e^{-\s\kappa t}\int_{0}^{t} \big(s^{-3\eps/2}+1\big) \big((t-s)^{-1/2}+1\big)e^{-\kappa(t-s)/2}  ds.
\end{equation*}
Now, by~\eqref{Int} we deduce
 \begin{equation*}
 \|   \Phi_{1}(h,G)(t)\|_{B^{}} \leq    C\eps_0(1+\|h\|_{X})(1+\|G\|_{Y}) \e^{-\s\kappa t},
 \end{equation*}
which in turn yields the bound
\begin{equation}\label{B10}
 \|  \Phi_{1}(h,G)\|_{X}\leq   C\eps_0(1+\|h\|_{X})(1+\|G\|_{Y})\leq C \eps_{0}\big(1+ \|  (h,G)\|_{Z}\big)^{2}.
\end{equation}

$\bullet$ We turn to the estimation of $\|\Phi_{2}(h,G)\|_{Y} $. We apply~\eqref{claim2} with
\begin{equation*}
h_{0}=\int_0^t \e^{-(t-s)K}  \big(F_{0}(s)+G(s)\big)\D_v  \big(\mminf+\e^{-sK}g_{0}+h(s)\big)ds,
\end{equation*}
then for all $t\geq 0$
\begin{eqnarray*}
\|\Phi_{2}(h,G)(t)\|_{L^{\infty}(\R^{2})}&\leq&  C \|\Lambda^{ \eps} \int_0^t \e^{-(t-s)K}  \big(F_{0}(s)+G(s)\big)\D_v \big(\mminf+\e^{-sK}g_{0}+h(s)\big)ds\|_{B^{}}\\
&\leq &  C\int_{0}^{t}   \| \Lambda^{ \eps}\e^{-(t-s)K}  \big(F_{0}+G\big)\D_v  \big(\mminf+\e^{-sK}g_{0}+h(s)\big)\|_{B^{}}ds\\
&\leq&  C\eps_{0} \int_{0}^{t}  \|F_{0}+G\|_{L^{\infty}(\R^{3})} \| \Lambda^{ \eps}\e^{-(t-s)K} \Lambda_{v} \|_{B^{\perp}}\|   \mminf+\e^{-sK}g_{0}+h(s) \|_{B^{}}ds,
\end{eqnarray*}
 where in the last line we used~\eqref{deriv}. Then by~\eqref{370},~\eqref{380} and~\eqref{prod1} with $\a=\eps$ and $\beta=0$ we get
\begin{multline*}
\|\Phi_{2}(h,G)(t)\|_{L^{\infty}(\R^{2})} \leq\\
\leq C \eps_{0}(1+\|h\|_{X})(1+\|G\|_{Y})    \e^{-\s \kappa t}  \int_{0}^{t} \big(s^{-3\eps/2}+1\big) \big((t-s)^{-1/2-3\eps/2}+1\big) e^{- \kappa (t-s)/2}  ds.
\end{multline*}
By~\eqref{Int}, this in turn implies
\begin{equation}\label{B30}
\|\Phi_{2}(h,G)\|_{Y} \leq C\eps_0(1+\|h\|_{X})(1+\|G\|_{Y})\leq C \eps_{0}\big(1+ \|  (h,G)\|_{Z}\big)^{2}.
\end{equation}

As a result, by~\eqref{B10}  and~\eqref{B30} there exists $C>0$ such that
\begin{equation*}
\|\Phi(h,G)\|_{Z} \leq  C \eps_{0}\big(1+ \|  (h,G)\|_{Z}\big)^{2}.
\end{equation*}
Therefore  we can  choose $\eps_{0}>0$ small enough so that  $\Phi$ maps the ball $\Gamma_{1}\subset Z$  into itself.\ligne

{\bf Step2: $\Phi$ is a contraction of $\Gamma_{1}$}\medskip

With exactly the  same arguments, we can also prove the contraction estimate
\begin{equation*}
\|  \Phi_{1}(h_{2},G_{2})-\Phi_{1}(h_{1},G_{2})\|_{Z} \leq C\eps_0\|  (h_{2}-h_{1},G_{2}-G_{1})\|_{Z}\big(1 +\|(h_{1},G_{1})\|_{Z}+\|(h_{2},G_{2})\|_{Z}\big).
\end{equation*}
We do not write the details.\ligne

As a conclusion, if $\eps_0>0$ is small enough, $\Phi$ has a unique fixed point in $\Gamma_{1}\subset Z$. This  shows \Bk the existence of a unique   $ h\in   \mathcal{C}\big(\,[0,+\infty[ \,;\,  B^{ }(\R^{4}) \big)$ such that $f=\mminf +\e^{-tK}g_{0}+h$ solves~\eqref{vpfp}.

\subsection{Conclusion of the proof of Theorem~\ref{thm1}}

The  convergence of $f$ to equilibrium follows from the choice of the weighted spaces. By definition
\begin{equation*}
\|h(t)\|_{B^{}}\leq C \e^{-\s \kappa t}\|h\|_{X}\longrightarrow 0,\quad \text{when}\quad  t\longrightarrow +\infty.
\end{equation*}
 Similarly,
 \begin{equation*}
\|G(t)\|_{L^{\infty}}\leq C \e^{-\s \kappa t}\|G\|_{Y}\longrightarrow 0,\quad \text{when}\quad  t\longrightarrow +\infty.
\end{equation*}

\section{\texorpdfstring{Proof of Theorem~\ref{thm2} (case $\boldsymbol {d=3}$)}{Proof of Theorem 1.3 (case d=3)}}\label{sect5}

\subsection{Small time analysis: $0\leq t\leq 1$} To begin with we prove a local well-posedness result for~\eqref{vpfp}.

\begin{prop} \ph \label{Prop51}
 Let $d=3$ and   $1/2<a<3/4$. Assume    that $f_{0}\in B^{a,a}(\R^6)\cap L^{\infty}(\R^{6}) $ is such that $U_{0}$ defined in \eqref{hyp} is in $W^{2,\infty}(\R^{3})$.  Assume moreover that Assumptions~\ref{assumption1} and~\ref{assumption2} are  satisfied. Then if $\vert \eps_0\vert$ is small enough, there exists  a unique local mild solution $f$ to~\eqref{vpfp} in the class
\begin{equation*}
 f\in \mathcal{C}\big(\,[0,1] \,;\,  B^{a,a}(\R^{6})  \big)\cap  L^{\infty}\big(\,[0,1] \,;\,    L^{\infty}(\R^{6})\big).
\end{equation*}
\end{prop}

We write
\begin{equation*}
f=\e^{-tK_{0}}f_{0}+g,\qquad E=E_{0}+ F,
\end{equation*}
where $U_{0},E_{0}$ and $F$ are defined by
\begin{equation*}
 \partial_x \Uzero=E_{0}=-  \frac{1}{|\S^{2}|}\frac{x}{|x|^{3}}\star \int f_{0}dv,\qquad F=-  \frac{1}{|\S^{2}|}\frac{x}{|x|^{3}}\star \int g dv.
\end{equation*}
In the regime $0\leq t\leq1$, the effective Fokker-Planck operator is given by
\begin{equation*}
K_{0} = v.\D_{x} -\D_x \Vzero(x).\D_{v}- \D_{v}.\left(\D_{v}+v\right),
\end{equation*}
where $\Vzero=V_{e}+\eps_{0}\Uzero$. The mild formulation of~\eqref{vpfp}, using $K_{0}$, is therefore
\begin{equation}\label{int2}
  \left\{
  \begin{aligned}
   &  f(t) = e^{-tK_{0}}f_0 +\eps_{0} \int_0^t e^{-(t-s)K_{0}}  (E(s)-E_{0})\D_v f(s)ds, \\
   & E(t)    = -  \frac{1}{|\S^{2}|} \frac{x}{|x|^3}\star_x \int f(t)dv.
    \end{aligned}
    \right.
\end{equation}

 We construct the solution with a  fixed point argument on $(g,F)$, and therefore we define the map   $\Phi=(\Phi_{1},\Phi_{2})$ given by
\begin{equation*}
\Phi_{1}(g,F)(t)= \eps_{0}\int_0^t \e^{-(t-s)K_{0}}   F(s) \D_v \big(\e^{-sK_{0}} f_{0}+g(s) \big)ds
\end{equation*}
\begin{equation*}
\Phi_{2}(g,F)(t)=-  \frac{1}{|\S^{2}|} \frac{x}{|x|^{3}}\star \int_{\R^{3}}\Big[ \big(e^{-tK_{0}}-1\big)f_{0}+\eps_{0}\int_0^t \e^{-(t-s)K_{0}}   F(s) \D_v \big( \e^{-sK_{0}}f_{0}+g(s) \big)ds\Big],
\end{equation*}
and we observe that $(f,E)$ solves~\eqref{int2} if and only if $(g,F)=\Phi(g,F)$.  For $\a,\b,\gamma  \geq \Bk0$ define the norms
\begin{equation*}
\|g\|_{X^{\a,\b}}=\sup_{0\leq t\leq1}  \|g(t,.)\|_{B_{x,v}^{\a,\b}(\R^{6})},
\end{equation*}
\begin{equation*}
\|F\|_{Y_{\gamma}}=\sup_{0\leq t\leq1} t^{-\gamma}\|F(t,.)\|_{L^{\infty}(\R^{3})},
\end{equation*}
define  the Banach space
$$Z:=X^{\a,\b}\times Y_{\gamma},\qquad \text{with}\qquad \|(h,G)\|_{Z}=\max\big(\|  h\|_{X^{\a,\b}},  \|G\|_{Y_{\gamma}}\big),$$
and   denote by $\Gamma_{R}$ the ball of radius $R$.

In the sequel we fix
\begin{equation*}
 \gamma={a/3-\eps},\qquad \a=a,\qquad \b=1,
\end{equation*}
for some $\eps\ll 1$.
\medskip

The end of this subsection is devoted to the proof of Proposition~\ref{Prop51}. We assume that $g_{0}\in B^{a,a}$, for some $a>1/2$. In the sequel, we write $K=K_{0}$.\medskip

{\bf Step1: $\Phi$ maps some ball $\Gamma_{R} \subset Z$ into itself}\medskip

$\bullet$ Firstly, we estimate $ \Phi_{1}(g,F)$ in $X^{0,1}$.
By~\eqref{deriv}, we have for all $0\leq t\leq 1$
\begin{eqnarray}
\| \Lambda_{v} \Phi_{1}(g,F)(t)\|_{B^{}}&\leq & \eps_{0}\int_{0}^{t}   \| \Lambda_{v}\e^{-(t-s)K}  F(s)\D_v \big(\e^{-sK}f_{0}+g(s)\big)\|_{B^{}}ds \nonumber\\
&\leq & C\eps_{0} \int_{0}^{t}  \|F\|_{L^{\infty}(\R^{3})} \| \Lambda_{v}\e^{-(t-s)K} \|_{B}\|  \Lambda_{v} \big(\e^{-sK}f_{0}+g(s)\big)\|_{B^{}}ds,\label{53}
\end{eqnarray}
and we estimate each factor in the previous integral thanks to the low regularity subsection results.\medskip

\underline{Estimation of $\|  \Lambda_{v} \big(\e^{-sK}f_{0}+g(s)\big)\|_{B^{}}$:} To begin with, we use point $(i)$ of Corollary~\ref{corlow} \Bk to estimate $\|  \Lambda_{v} \e^{-sK}f_{0} \|_{B^{}} $. Since $f_{0}\in B^{a,a}$ for some $a>1/2$, then for $\delta=1/2-a/2$ we have
\begin{equation*}
\|  \Lambda^{}_{v} \e^{-sK}f_{0} \|_{B^{}}\leq Cs^{-\delta}\|f_{0}\|_{B^{a}}.
\end{equation*}
 This gives \Bk for $0\leq s\leq t\leq 1$
\begin{eqnarray}
\|  \Lambda_{v} \big( \e^{-sK}f_{0}+g(s)\big)\|_{B^{}} &\leq & \|\e^{-sK}f_{0}\|_{B^{0,1}}+\|g(s) \|_{B^{0,1}}\nonumber \\
&\leq & C+C s^{-\delta}\|f_{0}\|_{B^{a}}+  \|g\|_{X^{a,1}} \nonumber \\
&\leq & C s^{-\delta}(1+\|g\|_{X^{a,1}}).\label{54}
\end{eqnarray}

\underline{Estimation of $ \|F \|_{L^{\infty}(\R^{3})}$:}
By   definition of the space $Y_{\gamma}$ we have
\begin{equation}\label{55}
  \|F \|_{L^{\infty}(\R^{3})} \leq    s^{\gamma}  \|F\|_{Y_{\gamma}}.
 \end{equation}
 \medskip

 \underline{Estimation of $  \| \Lambda_{v}\e^{-(t-s)K} \|_{B^{}\to B^{}}$:}
By  point $(ii)$ in Proposition~\ref{lowregularity}  we have
\begin{equation}\label{56}
   \| \Lambda^{}_{v}\e^{-(t-s)K} \|_{B^{}\to B^{}} \leq C(t-s)^{-1/2}.
 \end{equation}

Therefore by~\eqref{53},~\eqref{54},~\eqref{55} and~\eqref{56}, we have
 \begin{eqnarray*}
 \| \Lambda_{v} \Phi_{1}(g,F)(t)\|_{B^{}} \leq   C \eps_{0}(1+\|g\|_{X^{a,1}})\|F\|_{Y_{\gamma}}    \int_{0}^{t} s^{\gamma-\delta}(t-s)^{-1/2}     ds\nonumber \\
 \leq    C\eps_0(1+\|g\|_{X^{a,1}})\|F\|_{Y_{\gamma}} t^{\gamma-\delta+1/2}.
 \end{eqnarray*}
  As a consequence (using that $\gamma-\delta+1/2\geq 0$) we have proved
\begin{equation}\label{B1}
 \|  \Phi_{1}(g,F)\|_{X^{0,1}}\leq   C\eps_0(1+\|g\|_{X^{a,1}})\|F\|_{Y_{\gamma}}\leq C \eps_{0}\big(1+ \|  (g,F)\|_{Z}\big)^{2}.
\end{equation}

$\bullet$ We estimate $ \Phi_{1}(g,F)$ in $X_{}^{a,0}$. With the same arguments and the bound given in point $(iii)$ of Corollary~\ref{corlow}, for all $t\geq 0$ we obtain
 \begin{eqnarray}
 \| \Lambda^{a}_{x} \Phi_{1}(g,F)(t)\|_{B^{}}
& \leq& \eps_{0}\int_{0}^{t}   \| \Lambda^{a}_{x}\e^{-(t-s)K}  F(s)\D_v \big( \e^{-sK}f_{0}+g(s)\big)\|_{B}ds \label{phi1x}\\
&\leq & C\eps_{0} \int_{0}^{t}  \|F \|_{L^{\infty}(\R^{3})} \| \Lambda^{a}_{x}\e^{-(t-s)K}  \|_{B }\|  \Lambda^{}_{v} \big( \e^{-sK}f_{0}+g(s)\big)\|_{B^{}}ds\nonumber \\
&\leq & C \eps_{0}(1+\|g\|_{X_{}^{a,1}})\|F\|_{Y_{\gamma}}   \int_{0}^{t} s^{\gamma-\delta}\big((t-s)^{ -3a/2} +1\big)   ds \nonumber\\
&\leq  & C \eps_{0}(1+\|g\|_{X_{}^{a,1}})\|F\|_{Y_{\gamma}}     \big(t^{\gamma-\delta+1-3a/2}+t^{\gamma-\delta+1}\big).\nonumber
\end{eqnarray}
 This  in turn implies (observing that $\gamma-\delta+1-3a/2>0$ provided that $a<3/4$)
\begin{equation}\label{B2}
 \|  \Phi_{1}(g,F)\|_{X_{}^{a,0}}\leq  C\eps_0(1+\|g\|_{X_{}^{a,1}})\|F\|_{Y_{\gamma}}\leq C \eps_{0}\big(1+ \|  (g,F)\|_{Z}\big)^{2}.
\end{equation}
\medskip

$\bullet$ We turn to the estimation of $\|\Phi_{2}(g,F)\|_{Y_{\gamma}} $. We apply~\eqref{211} and~\eqref{claim} with
\begin{equation*}
h_{0}=\big(e^{-tK}-1\big)f_{0}+\eps_{0}\int_0^t \e^{-(t-s)K}   F(s) \D_v  \big( \e^{-sK}f_{0}+g(s)\big)ds,
\end{equation*}
then for all $0\leq t\leq 1$
\begin{equation*}
\|\Phi_{2}(g,F)(t)\|_{L^{\infty}(\R^{3})} \leq  C_{0}  t^{a/3-\eps} + C\eps_{0}\int_{0}^{t}   \| \Lambda^{1/2+\eps}\e^{-(t-s)K}   F(s)\D_v  \big( \e^{-sK}f_{0}+g(s)\big)\|_{B^{}}ds,
\end{equation*}
where we used Lemma~\ref{convolmoins}.

  To control the second term, we can proceed as in~\eqref{phi1x} with $a$ replaced by $1/2+\eps$. Actually we have
\begin{equation*}
\|\Phi_{2}(g,F)(t)\|_{L^{\infty}(\R^{3})} \leq
 C_{0}  t^{a/3-\eps} + C\eps_{0} \int_{0}^{t}  \|F \|_{L^{\infty}(\R^{3})} \| \Lambda^{1/2+\eps}\e^{-(t-s)K}  \|_{B^{}}\|  \Lambda^{}_{v} \big( \e^{-sK}f_{0}+g(s)\big)\|_{B^{}}ds,
\end{equation*}
 and we get
\begin{eqnarray*}
\|\Phi_{2}(g,F)(t)\|_{L^{\infty}(\R^{3})} &\leq& C_{0}  t^{a/3-\eps} +C \eps_{0}(1+\|g\|_{X_{}^{a,1}})\|F\|_{Y_{\gamma}}      \int_{0}^{t} s^{\gamma-\delta}  \big((t-s)^{-\frac32(\frac12+\eps)}+1\big)   ds \\
&\leq &C_{0}  t^{a/3-\eps} +C \eps_{0}(1+\|g\|_{X_{}^{a,1}})\|F\|_{Y_{\gamma}}       \big(      t^{\gamma-\delta+1/4-3\eps/2}+ t^{\gamma-\delta+1}   \big)\\
&\leq &    t^{\gamma}\big[ C_{0}+C \eps_{0}(1+\|g\|_{X_{}^{a,1}})\|F\|_{Y_{\gamma}}     (      t^{ -\delta+1/4-3\eps/2}+ t^{ -\delta+1}   )\big],
 \end{eqnarray*}
 since $\gamma=a/3-\eps$. Therefore
\begin{equation}\label{B3}
\|\Phi_{2}(g,F)\|_{Y_{\gamma}} \leq C_{0}+C\eps_0(1+\|g\|_{X_{}^{a,1}})\|F\|_{Y_{\gamma}}\leq C \eps_{0}\big(1+ \|  (g,F)\|_{Z}\big)^{2},
\end{equation}
provided that $\delta<1$ and  $1/2-a/2=\delta <1/4-3\eps/2$. This latter condition can be  satisfied for $\eps>0$ small enough, since $a>1/2$.\medskip

As a result, by~\eqref{B1},~\eqref{B2} and~\eqref{B3} there exists $C>0$ such that
\begin{equation*}
\|\Phi(g,F)\|_{Z} \leq  C_{0}+C \eps_{0}\big(1+ \|  (g,F)\|_{Z}\big)^{2}.
\end{equation*}
Therefore  we can  choose $\eps_{0}>0$ small enough so that  $\Phi$ maps the ball $\Gamma_{2C_{0}}\subset Z$  into itself.\ligne

{\bf Step2: $\Phi$ is a contraction of $\Gamma_{2C_{0}}$}\medskip

With exactly the  same arguments, we can also prove the contraction estimate
\begin{equation*}
\|  \Phi_{1}(g_{2},F_{2})-\Phi_{1}(g_{1},F_{1})\|_{Z} \leq C\eps_0\|  (g_{2}-g_{1},F_{2}-F_{1})\|_{Z}\big(1 +\|(g_{1},F_{1})\|_{Z}+\|(g_{2},F_{2})\|_{Z}\big).
\end{equation*}
We do not write the details.\ligne

As a conclusion, if $\eps_0>0$ is small enough, $\Phi$ has a unique fixed point in $\Gamma_{2C_{0}}\subset Z$. This  shows \Bk the existence of a unique   $ g\in   \mathcal{C}\big(\,[0,1] \,;\,  B^{a,1}(\R^{6}) \big)$  such that $f= \e^{-tK}f_{0}+g$ solves~\eqref{vpfp}.

\subsection{Long time analysis: $t\in ]0,+\infty[$} We now study long time existence and trend to equilibrium. We use here the spaces defined in the Subsection~\ref{Sub41}. Let     $1/2<\b<1$ and $0<\a<2/3$   be such that $\a<(1+\beta)/3$. Fix also
\begin{equation*}
0<\delta<\beta/2-1/4, \qquad 0<\s\leq \frac12 \min\big(1-\b+\a,1\big),
\end{equation*}
 which is realised for, say, $\s=1/12$. From now, we assume that all these conditions are satisfied. \medskip

 In this section we prove the following result
\begin{prop} \ph \label{Prop52}
Let $d=3$. Assume that    $1/2<a<2/3$ and   that $f_{0}\in B^{a,a}(\R^6) $. Assume moreover that Assumptions~\ref{assumption1} and~\ref{assumption2} are  satisfied. Then if $\vert \eps_0\vert$ is small enough, there exists  a unique local mild solution $f$ to~\eqref{vpfp} which reads
\begin{equation*}
 f(t)=\mminf+\e^{-tK}(f_{0}-\mminf)+h(t)
\end{equation*}
where $h\in X^{\a,\b}_{\delta}$, thus
\begin{equation*}
 h\in \mathcal{C}\big(\,]0,+\infty[ \,;\,  B^{\a,\b}(\R^{6}) \big).
\end{equation*}
\end{prop}

Since $f_{0}-\mminf\in B^{\perp}$, and by definition of the space $X^{\a,\b}_{\delta}$, we obtain the exponentially fast convergence of $f$ to $\mminf$. Notice that in the previous result, the parameters $(\a,\b)$ can be chosen independently from $a$.  If \Bk one chooses  $\a=a$ and $\beta$  close to 1, then the result of Proposition~\ref{Prop52} combined with Proposition~\ref{Prop51} and Corollary~\ref{lemMax1} implies Theorem~\ref{thm2}.\medskip

We now turn to the proof  of Proposition~\ref{Prop52}.  Let $g_{0}:=f_{0}-\mminf \in B^{a,a}\cap B^{\perp}$, for some $a>1/2$.
We denote by $\Gamma_{1}$ the unit ball in $Z$, and in the sequel, we use the same notations and decomposition as in Section~\ref{Sub41}.  Then \medskip

{\bf Step1: $\Phi$ maps the ball $\Gamma_{1} \subset Z$ into itself}\medskip

$\bullet$ Firstly, we estimate $ \Phi_{1}(h,G)$ in $X_{\delta}^{0,\b}$.
By~\eqref{deriv}, we have for all $t\geq 0$
\begin{multline}
\| \Lambda^{\b}_{v} \Phi_{1}(h,G)(t)\|_{B^{}}\leq  \eps_{0}\int_{0}^{t}   \| \Lambda^{\b}_{v}\e^{-(t-s)K}  \big(F_{0}+G\big)\D_v \big(\mminf+\e^{-sK}g_{0}+h(s)\big)\|_{B^{}}ds  \\
\leq  C\eps_{0} \int_{0}^{t}  \|F_{0}+G\|_{L^{\infty}(\R^{3})} \| \Lambda^{\b}_{v}\e^{-(t-s)K} \Lambda^{1-\b}_{v} \|_{B^{\perp}}\|  \Lambda^{\b}_{v} \big(\mminf+\e^{-sK}g_{0}+h(s)\big)\|_{B^{}}ds,\label{214}
\end{multline}
and we estimate each factor in the previous integral.\medskip

\underline{Estimation of $\|  \Lambda^{\b}_{v} \big(\mminf+\e^{-sK}g_{0}+h(s)\big)\|_{B^{}}$:} To begin with, we use Lemma~\ref{leminterpol} to estimate the term $\|  \Lambda^{\b}_{v} \e^{-sK}g_{0} \|_{B^{}} $. Since $g_{0}\in B^{a,a}$ for some $a>1/2$, there exists $0<\delta<\beta/2-1/4$ such that
\begin{equation*}
\|  \Lambda^{\b}_{v} \e^{-sK}g_{0} \|_{B^{}}\leq C(1+s^{-\delta})\|g_{0}\|_{B^{a}}.
\end{equation*}
Then, we use that  $\mminf \in B^{\a,\b}$, and by~\eqref{dv} we obtain
\begin{eqnarray}
\|  \Lambda^{\b}_{v} \big(\mminf+\e^{-sK}g_{0}+h(s)\big)\|_{B^{}} &\leq & \|  \mminf \|_{B^{0,\beta}} +\|\e^{-sK}g_{0}\|_{B^{0,\beta}}+\|h(s) \|_{B^{0,\beta}}\nonumber \\
&\leq & C+C (1+s^{-\delta})\|g_{0}\|_{B^{a}}+ (1+s^{-\delta})\|h\|_{X_{\delta}^{\a,\b}} \nonumber \\
&\leq & C (1+s^{-\delta})(1+\|h\|_{X_{\delta}^{\a,\b}}).\label{37}
\end{eqnarray}

\underline{Estimation of $ \|F_{0}+G\|_{L^{\infty}(\R^{3})}$:}
By~\eqref{e000} and the definition of the space $Y_{ }$ we get for  $\eps>0$ small enough and $a>1/2$
\begin{eqnarray}\label{38}
  \|F_{0}+G\|_{L^{\infty}(\R^{3})}&\leq&   \|F_{0}\|_{L^{\infty}(\R^{3})} + \|G\|_{L^{\infty}(\R^{3})}\nonumber \\
  &\leq&   C  e^{-\kappa s/2}\|g_{0}\|_{B^{a}} + C \e^{-\s\kappa s}\|G\|_{Y_{ }}\nonumber\\
&\leq &    C  \e^{-\s\kappa s}(1+\|G\|_{Y_{ }}).
 \end{eqnarray}

 \underline{Estimation of $  \| \Lambda^{\b}_{v}\e^{-(t-s)K} \Lambda^{1-\b}_{v} \|_{B^{\perp}\to B^{\perp}}$:}
This is exactly~\eqref{prod2}, namely
\begin{equation*}
 \| \Lambda^{\b}_{v}\e^{-(t-s)K} \Lambda^{1-\b}_{v} \|_{B^{\perp}\to B^{\perp}} \leq  C\big((t-s)^{-1/2}+1\big)e^{-\kappa (t-s)}.
\end{equation*}
 ~

Therefore by~\eqref{214},~\eqref{37},~\eqref{38}, we have when $ \delta<1$
 \begin{multline*}
 \| \Lambda^{\b}_{v} \Phi_{1}(h,G)(t)\|_{B^{}} \leq \\
\begin{aligned}
&\leq C \eps_{0}(1+\|h\|_{X_{\delta}^{\a,\b}})(1+\|G\|_{Y_{ }})    \e^{-\s\kappa t}\int_{0}^{t} \big(s^{ -\delta}+1\big) \big((t-s)^{-1/2}+1\big)e^{-\kappa(t-s)/2}      ds
\end{aligned}
\end{multline*}
Now, by~\eqref{Int}, if we denote by $\eta=-\min({1/2 -\delta},0)$, we have
 \begin{eqnarray}\label{311}
 \| \Lambda^{\b}_{v} \Phi_{1}(h,G)(t)\|_{B^{}} \leq    C\eps_0(1+\|h\|_{X_{\delta}^{\a,\b}})(1+\|G\|_{Y_{ }}) \e^{-\s\kappa t} \big({\bf 1}_{\{0< t\leq 1\}}t^{-\eta}+1\big)\nonumber \\
 \leq    C\eps_0(1+\|h\|_{X_{\delta}^{\a,\b}})(1+\|G\|_{Y_{ }}) \e^{-\s\kappa t} \big(t^{-\delta}+1\big).
 \end{eqnarray}
  As a consequence we have proved
\begin{equation}\label{515}
 \|  \Phi_{1}(h,G)\|_{X_{\delta}^{0,\beta}}\leq   C\eps_0(1+\|h\|_{X_{\delta}^{\a,\b}})(1+\|G\|_{Y_{ }})\leq C \eps_{0}\big(1+ \|  (h,G)\|_{Z}\big)^{2}.
\end{equation}

$\bullet$ We estimate $ \Phi_{1}(h,G)$ in $X_{\delta}^{\a,0}$. With the same arguments and the bound~\eqref{prod1}, for all $t\geq 0$ we obtain
 \begin{multline}
 \| \Lambda^{\a}_{x} \Phi_{1}(h,G)(t)\|_{B^{}}\leq\\
\begin{aligned}
& \leq \eps_{0}\int_{0}^{t}   \| \Lambda^{\a}_{x}\e^{-(t-s)K}  \big(F_{0}+G\big)\D_v \big(\mminf+\e^{-sK}g_{0}+h(s)\big)\|_{B^{}}ds \label{516}\\
&\leq  C\eps_{0} \int_{0}^{t}  \|F_{0}+G\|_{L^{\infty}(\R^{3})} \| \Lambda^{\a}_{x}\e^{-(t-s)K} \Lambda^{1-\b}_{v} \|_{B^{\perp}}\|  \Lambda^{\b}_{v} \big(\mminf+\e^{-sK}g_{0}+h(s)\big)\|_{B^{}}ds\\
&\leq  C \eps_{0}(1+\|h\|_{X_{\delta}^{\a,\b}})(1+\|G\|_{Y_{ }})    \e^{-\s\kappa t} \int_{0}^{t} \big(s^{ -\delta}+1\big) \big((t-s)^{-1/2+\b/2-3\a/2} +1\big)   e^{-\kappa (t-s)/2}    ds \\
&\leq   C \eps_{0}(1+\|h\|_{X_{\delta}^{\a,\b}})(1+\|G\|_{Y_{ }})    \e^{-\s\kappa t}\big(t^{-\delta}+1\big),
\end{aligned}
\end{multline}
using~\eqref{Int} and the fact that $1/2+\b/2-3\a/2>0$.

 This  in turn implies
\begin{equation}\label{517}
 \|  \Phi_{1}(h,G)\|_{X_{\delta}^{\a,0}}\leq  C\eps_0(1+\|h\|_{X_{\delta}^{\a,\b}})(1+\|G\|_{Y_{ }})\leq C \eps_{0}\big(1+ \|  (h,G)\|_{Z}\big)^{2}.
\end{equation}
\medskip

$\bullet$ We turn to the estimation of $\|\Phi_{2}(h,G)\|_{Y_{ }} $. We apply~\eqref{claim} with
\begin{equation*}
h_{0}=\eps_{0}\int_0^t \e^{-(t-s)K}  \big(F_{0}(s)+G(s)\big)\D_v  \big(\mminf+\e^{-sK}g_{0}+h(s)\big)ds,
\end{equation*}
then for all $t\geq 0$
\begin{eqnarray*}
\|\Phi_{2}(h,G)(t)\|_{L^{\infty}(\R^{3})}&\leq&  C\eps_{0} \|\Lambda^{1/2+\eps} \int_0^t \e^{-(t-s)K}  \big(F_{0}(s)+G(s)\big)\D_v \big(\mminf+\e^{-sK}g_{0}+h(s)\big)ds\|_{B^{}}\\
&\leq &  C\eps_{0}\int_{0}^{t}   \| \Lambda^{1/2+\eps}\e^{-(t-s)K}  \big(F_{0}+G\big)\D_v  \big(\mminf+\e^{-sK}g_{0}+h(s)\big)\|_{B^{}}ds.
\end{eqnarray*}
Now we can proceed as in~\eqref{516} with $\a$ replaced by $1/2+\eps$. Actually we have
\begin{multline*}
\|\Phi_{2}(h,G)(t)\|_{L^{\infty}(\R^{3})}\leq  \\
\leq C\eps_{0} \int_{0}^{t}  \|F_{0}+G\|_{L^{\infty}(\R^{3})} \| \Lambda^{1/2+\eps}\e^{-(t-s)K} \Lambda^{1-\b}_{v} \|_{B^{\perp}}\|  \Lambda^{\b}_{v} \big(\mminf+\e^{-sK}g_{0}+h(s)\big)\|_{B^{}}ds,
\end{multline*}
and for $1/2<\beta<1$,  $\eps\ll 1$, by~\eqref{prod1} and~\eqref{Int} we get
\begin{multline*}
\|\Phi_{2}(h,G)(t)\|_{L^{\infty}(\R^{3})} \leq\\
\begin{aligned}
&\leq C \eps_{0}(1+\|h\|_{X_{\delta}^{\a,\b}})(1+\|G\|_{Y_{ }})    \e^{-\s \kappa t}  \int_{0}^{t} \big(s^{ -\delta}+1\big) \big((t-s)^{-\frac12+\frac\b2-\frac32(\frac12+\eps)}+1\big)   e^{-\kappa (t-s)/2}       ds \\
&\leq C \eps_{0}(1+\|h\|_{X_{\delta}^{\a,\b}})(1+\|G\|_{Y_{ }})    \e^{-\s\kappa t}       \big(    {\bf 1}_{\{0<t\leq 1\}}  t^{-1/4+\b/2-\delta-3\eps/2 }+ 1   \big)\\
&\leq C \eps_{0}(1+\|h\|_{X_{\delta}^{\a,\b}})(1+\|G\|_{Y_{ }})    \e^{-\s\kappa t},
\end{aligned}
\end{multline*}
if $\eps>0$ is chosen  small enough so that we have $\delta<\b/2-1/4-3\eps/2$. This in turn implies
\begin{equation}\label{518}
\|\Phi_{2}(h,G)\|_{Y_{ }} \leq C\eps_0(1+\|h\|_{X_{\delta}^{\a,\b}})(1+\|G\|_{Y_{ }})\leq C \eps_{0}\big(1+ \|  (h,G)\|_{Z}\big)^{2}.
\end{equation}

As a result, by~\eqref{515},~\eqref{517} and~\eqref{518} there exists $C>0$ such that
\begin{equation*}
\|\Phi(h,G)\|_{Z} \leq  C \eps_{0}\big(1+ \|  (h,G)\|_{Z}\big)^{2}.
\end{equation*}
Therefore  we can  choose $\eps_{0}>0$ small enough so that  $\Phi$ maps the ball $\Gamma_{1}\subset Z$  into itself.\ligne

{\bf Step2: $\Phi$ is a contraction of $\Gamma_{1}$}\medskip

With exactly the  same arguments, we can also prove the contraction estimate
\begin{equation*}
\|  \Phi_{1}(h_{2},G_{2})-\Phi_{1}(h_{1},G_{2})\|_{Z} \leq C\eps_0\|  (h_{2}-h_{1},G_{2}-G_{1})\|_{Z}\big(1 +\|(h_{1},G_{1})\|_{Z}+\|(h_{2},G_{2})\|_{Z}\big).
\end{equation*}
We do not write the details.\ligne

As a conclusion, if $\eps_0>0$ is small enough, $\Phi$ has a unique fixed point in $\Gamma_{1}\subset Z$. This  shows \Bk  the existence of a unique   $ h\in   \mathcal{C}\big(\,]0,+\infty[ \,;\,  B^{\alpha,\beta}(\R^{6}) \big)$  such that $f=\mminf +\e^{-tK}g_{0}+h$ solves~\eqref{vpfp}. \medskip



\begin{thebibliography}{99}

 \bibitem{BDP}
A. Blanchet, J. Dolbeault and B. Perthame.
\newblock Two-dimensional Keller-Segel model: optimal critical mass and qualitative properties of the solutions.
\newblock {\em Electron. J. Differential Equations} 2006, No. 44, 32 pp.


 \bibitem{Bou93}
 F. Bouchut.
\newblock Existence and uniqueness of a global smooth solution for the
Vlasov-Poisson-Fokker-Planck system in three dimensions.
\newblock {\em J. Funct. Anal.} \textbf{111}, (1993), 239--258.


\bibitem{BD95}
F. Bouchut and J. Dolbeault.
\newblock  On long time asymptotics of the Vlasov-Fokker-Planck equation and of the Vlasov-Poisson-Fokker-Planck system with Coulombic and
Newtonian potentials.
\newblock {\em Diff. Int. Eq.}, \textbf{8}, no 3, (1995), 487--514.


\bibitem{CSV96}
J.A. Carrillo, J.   Soler and J.L.  Vasquez.
\newblock  Asymptotic behavior and selfsimilarity for the three dimentional Vlasov-Poisson-Fokker-Planck system.
\newblock {\em  J. Funct. Anal.}  \textbf{141}, (1996), 99--132.


\bibitem{Car98}
 A. Carpio.
\newblock    Long-time behaviour for solutions of the Vlasov-Poisson-Fokker-Planck equation.
\newblock {\em Math. Methods Appl. Sci.}  \textbf{21} (1998), no. 11, 985--1014.


 \bibitem{Deg86}
 P.   Degond.
\newblock   Global existence of smooth solutions for the Vlasov-Poisson-Fokker-Planck equation in 1 and~2 space dimensions.
\newblock {\em Ann. scient. Ec. Norm. Sup.}, (4)  \textbf{19} (1986), no. 4, 519--542


\bibitem{DV01}
L.  Desvillettes and C. Villani.
\newblock  On the trend to global equilibrium in spatially inhomogeneous systems. Part~I: the linear Fokker-Planck equation.
\newblock {\em Comm. Pure Appl. Math.}  \textbf{54}, 1 (2001), 1--42.


\bibitem{DV05}
L.  Desvillettes and C. Villani.
\newblock  On the trend to global equilibrium for spatially inhomogeneous systems: the Boltzmann equation.
\newblock {\em Invent. Math.}  \textbf{159}, No. 2,  (2005), 245--316.


 \bibitem{Dol91}
 J.  Dolbeault.
\newblock Stationary states in plasma physics: Maxwellian solutions of the Vlasov-Poisson system.
\newblock {\em J. Math. Models Methods Appl. Sci.}  \textbf{1} (1991), no. 2, 183–208.


 \bibitem{Dol99}
 J.  Dolbeault.
\newblock Free energy and solutions of the Vlasov-Poisson-Fokker-Planck system: external potential and confinment (large time behavior and
steady states).
\newblock {\em J. Math. Pures. Appl.}, \textbf{78}, (1999), 121--157.

 \bibitem{EN00} 
K.-J. Engel and R. Nagel.
\newblock{One parameter semigroups for linear evolution equations}. 
 Graduate Texts in Mathematics, 194. Springer-Verlag, New-York, 2000. \Bk



\bibitem{Guo02}
Y. Guo.
\newblock  The Landau equation in a periodic box.
\newblock {\em  Commun. Math. Phys.}, \textbf{231}, (2002), 391--434.


\bibitem{Guo03a}
Y. Guo.
\newblock Classical solutions to the Boltzmann equation for molecules with an angular cutoff.
\newblock {\em  Arch. Ration. Mech. Anal.}  \textbf{169}, No. 4, (2003), 305--353.


\bibitem{Guo03b}
Y. Guo.
\newblock The Vlasov-Maxwell-Boltzmann system near Maxwellians.
\newblock {\em Invent. Math.}  \textbf{153}, No.3, (2003), {593--630}.


\bibitem{Guo04}
Y. Guo.
\newblock The Boltzmann equation in the whole space.
\newblock {\em Indiana Univ. Math. J.}  \textbf{53}, No. 4, (2004), 1081--1094.


\bibitem{HelN04}
B. Helffer and F.  Nier.
\newblock  Hypoelliptic estimates and spectral theory for Fokker-Planck operators and Witten Laplacians.
\newblock {\em Lecture Notes in Mathematics}, 1862. Springer-Verlag, Berlin, 2005. x+209 pp.


 \bibitem{her06}
F. H\'erau.
\newblock  Short and long time behavior of the Fokker-Planck equation in a confining potential and applications.
\newblock {\em J. Funct. Anal.},  \textbf{244}, no. 1, 95--118 (2007).

\bibitem{HN04}
F. H{\'e}rau   and F. Nier.
\newblock   Isotropic hypoellipticity and trend to equilibrium for the  Fokker-Planck equation with high degree potential.
\newblock {\em Arch. Ration. Mech. Anal.},  \textbf{171} (2), 151--218, 2004. Announced in Actes colloque EDP Forges-les-eaux, 12p.,  (2002).


\bibitem{Hor95}
L. H\"ormander.
\newblock Symplectic classification of quadratic forms, and general Mehler formulas. \newblock{\em Mathematische Zeitschrift},  \textbf{219} (3), 413--450, 1995



\bibitem{HJ13}
H.J.  Hwang and J. Jang.
\newblock On the Vlasov-Poisson-Fokker-Planck equation near Maxwellian.
\newblock {\em Discrete Contin. Dyn. Syst. Ser. B.}   \textbf{18}  (2013),  no. 3, 681--691.


\bibitem{Kag01}
Y. Kagei.
\newblock  Invariant manifolds and long-time asymptotics for the Vlasov-Poisson-Fokker-Planck equation.
\newblock {\em SIAM J. Math. Anal.} \textbf{33} no. 2, (2001), 489--507.


\bibitem{Ler14} {N. Lerner.}
A course on integration theory, including more than 150 exercises with detailed answers, Birk\"auser, 2014.


\bibitem{OV90}
H.D. Victory and B.P. O'Dwyer.
\newblock   On classical solutions of Vlasov-Poisson-Fokker-Planck systems.
\newblock {\em  Indiana Univ.  Math. J.} \textbf{39}, (1990), 105--157.


\bibitem{OS00}
K. Ono   and W. Strauss.
\newblock Regular solutions of the Vlasov-Poisson-Fokker-Planck system.
\newblock {\em  Discrete Contin. Dyn. Syst.} \textbf{6}, no. 4, (2000), 751--772.


\bibitem{Paz83}
A. Pazy.
\newblock  Semigroups of linear operators and applications to partial differential equations,
\newblock Springer-Verlag, Berlin, second edition, 1983.



\bibitem{RW92}
 G. Rein and J. Weckler.
\newblock Generic global classical solutions of the Vlasov-Fokker-Planck-Poisson system in three dimensions.
\newblock {\em J. Diff. Eq.} \textbf{99}, (1992), 59--77.




\bibitem{Sol97}
J. Soler.
\newblock  Asymptotic behavior for the Vlasov-Poisson-Fokker-Planck system. Proceedings of the Second World Congress of Nonlinear Analysts, Part 8 (Athens, 1996).
\newblock {\em  Nonlinear Anal.} \textbf{30} no. 8, (1997), 5217--5228.

\bibitem{Ste70}
 E.M. Stein.
\newblock {\em Singular integrals and differentiability of
functions}.
\newblock Princeton University Press, Princeton, 1970.


 \bibitem{villani}
C. Villani.
\newblock Hypocoercivity.
\newblock {\em Mem. Amer. Math. Soc.} 202 (2009), no. 950, iv+141 pp.


\bibitem{Won99}
M.W. Wong.
\it An introduction to pseudo-differential operators. \rm World
Scientific, second edition, 1999.


 \end{thebibliography}
\end{document}